%% file: Regularity_properties_of_stationary_harmonic_functions_whose_Laplacian_is_a_Radon_measure.tex
\documentclass[12pt]{amsart}
\usepackage[latin1]{inputenc}

\usepackage{mathrsfs}

\usepackage{color}

\usepackage{amsfonts}
\usepackage{amssymb}

\usepackage{amsmath}
\usepackage{amsthm}

\usepackage{geometry}
\usepackage{graphicx}
\newtheorem{thm}{Theorem}[section]
\newtheorem{proposition}{Proposition}[section]
\newtheorem{lemma}{Lemma}[section]
\newtheorem{corollary}{Corollary}[section]
\theoremstyle{Definition}
\newtheorem{definition}{Definition}[section]

\newtheorem{exa}{Example}[section]
\newtheorem{claim}{Claim}[section]

\DeclareMathOperator{\dist}{dist}
\DeclareMathOperator{\dive}{div}
\DeclareMathOperator{\tra}{tr}
\DeclareMathOperator{\curl}{curl}
\DeclareMathOperator{\supp}{supp}
\DeclareMathOperator{\capa}{cap}
\DeclareMathOperator{\inte}{int}
\DeclareMathOperator{\ang}{ang}

\def\R{\mathbb R}
\def\C{\mathbb C}

\def \H {\mathcal{H}}

\def \e {\varepsilon}
\def \h {h_{\mu}}
\def \T {T_{\mu}}
\author{Rémy Rodiac}
\title[Regularity of some stationary harmonic functions]{Regularity properties of stationary harmonic functions whose Laplacian is a Radon measure}
\address{D\'epartement de Math\'ematiques, Universit\'e Paris-Est-Cr\'eteil, 61 avenue du G\'en\'eral de Gaulle, 94010 Cr\'eteil Cedex, France}
\email{  remy.rodiac@u-pec.fr}
\date{}
\begin{document}

\begin{abstract}
We study the regularity of Radon measures $\mu$ which satisfy that there exists a function $\h$ in $H^1(\Omega)$, 
stationary harmonic such that $\Delta \h =\mu$ in $\Omega$ (here $\Omega$ is an open set of $\R^2$). Such conditions appear in physical contexts such as the study of 
a limiting vorticity measure associated to a family $(u_\e)_\e$ of solutions of the Ginzburg-Landau system 
without magnetic field. Under these conditions we prove that locally there exists a harmonic function $H$ such 
that the support of the measure is contained in the set of zeros of $H$. Using the local structure of the set of 
zeros of harmonic functions we can thus obtain that locally the support of $\mu$ is a union of smooth simple 
curves.
\end{abstract}
\keywords{Stationary harmonic functions, Radon measure, Ginzburg-Landau system, Euler system, System of point vortices}
\subjclass{Primary 58E50; Secondary 35J20-82D55}
\maketitle
\section{Introduction and main results}

	Stationary harmonic functions arise in many physical problems such as the study of Ginzburg-Landau equations linked to superconductivity or the study of Euler equations in fluid mechanics. They are also related to limiting vorticities of stationary system of point vortices. Let $\Omega$ be a bounded open set in $\R^2$. 

\begin{definition}
A function $h$ in $H^1(\Omega)$ is \textit{stationary harmonic} if $\dive T_h=0$ in $\Omega$ in the sense of distributions, where $T_h$ is the stress-energy tensor associated to the Dirichlet energy, defined by 
\begin{equation} \label{stress-energy}
T_h=\begin{pmatrix}
\frac{1}{2}\left[(\partial_y h) ^2-(\partial_x h^2)\right] & -\partial_x h \partial_y h \\
-\partial_x h \partial_y h & \frac{1}{2}\left[(\partial_x h) ^2-(\partial_y h^2)\right]
\end{pmatrix}.
\end{equation}
Equivalently $h$ is stationary harmonic in $\Omega$ if 
\begin{equation}\label{2bis}
\omega_{h}:= (\partial_x h) ^2 -(\partial_y h)^2 -2i\partial_x h \partial_y h \text{ is holomorphic in } \Omega.
\end{equation}
\end{definition}


\noindent Equation \eqref{stress-energy} means that $\partial_x (T_h)_{i1} +\partial_y(T_h)_{i2}=0$ for  $i=1,2$ in the sense of distributions.
 Let us denote by  $H^{-1}(\Omega)$ the dual of the Sobolev space $H_0^{1}(\Omega)$. The aim of this paper is to describe the local regularity of Radon measures $\mu$ which satisfy the following conditions:

\begin{equation}\label{1}
\mu \in H^{-1}(\Omega),
\end{equation} 

\noindent there exists a function $\h$ such that 

\begin{equation}\label{3}
\Delta \h = \mu \ \text{in} \ \Omega,
\end{equation}

\noindent and 
\begin{equation}\label{2}
\h \text{ is stationary harmonic}.
\end{equation}

Note that if $\h$ is a solution of \eqref{3} then $\h \in H^1(\Omega)$ and then condition \eqref{2} is well-defined. Indeed we can see that there exists a solution of \eqref{3} in $H^1_0(\Omega)$ using the Lax-Milgram theorem. Then all the solutions are in $H^1(\Omega)$ since the difference between two solutions is harmonic in $\Omega$ end hence belongs to $H^1(\Omega)$. \\

We will discuss the physical motivations of this problem in the next section. Now we wish to examine in slightly more details the condition \eqref{2} and some of its direct consequences.
%
One can show that if $h$ is harmonic ($\Delta h =0$) then $h$ is stationary harmonic but the converse is not true in general. It 
is true if $h$ is regular. Indeed using the same techniques as in \cite{SandierSerfaty} chapter 13 we can prove 
that if $\mu$ is in $L^p$ for some $p >1$ then a solution of \eqref{1}, \eqref{3},\eqref{2} is harmonic, \textit{i.e.}, 
$\mu=0$. For the proof of these facts and other properties of stationary harmonic functions we refer to the Appendix.\\

\noindent Another direct consequence of condition \eqref{2} is that $\nabla \h \in L^\infty_\text{loc}$ and then $\h$ is 
locally lipschitz continuous. This is due to the fact that $|\nabla \h|^2 =|\omega_{\h}|^2$ and $\omega_{\h}$ is 
holomorphic in $\Omega$. In particular $\h$ and $|\nabla \h |$ are continuous. The fact that $\omega_{\h}$ is holomorphic also gives us the following:

\begin{proposition}
Let $\h$ which satisfies that $\omega_{\h}=(\partial_x \h) ^2 -(\partial_y \h)^2 -2i\partial_x \h \partial_y \h$ is holomorphic. Then the zeros of $\omega_{\h}$  are isolated in $\Omega$. If $\Omega$ is compact there is a finite number of such critical points. 
\end{proposition}

 In the present paper we are interested in describing the properties of Radon measures $\mu$ which satisfy hypothesis \eqref{1}, \eqref{3}, \eqref{2}. Let us recall that the support of a measure $\mu$ is the complement of the largest open set $A$ such that $\mu(A)=0$.  Our first result describes the local regularity of the measure $\mu$ in the neighborhood of point $z_0$ which belongs to the support of $\mu$ and such that $\omega_{\h}(z_0) \neq 0$. Note that we can always assume that $\h(z_0)=0$ because adding a constant to $h$ does not change the hypothesis \eqref{1}, \eqref{3}, \eqref{2}. Note also that near a point $z_0$ which does not belong to the support of $\mu$ the function $\h$ is a harmonic function.

\begin{thm}\label{theorem1}
Let $z_0 \in \supp \mu $, with $(\h,\mu)$ which satisfy assumptions \eqref{1}, \eqref{3}, \eqref{2} and such that $\omega_{\h}(z_0)\neq 0$. We assume that $\h(z_0)=0$. Then there exist a neighborhood $V$ of $z_0$ and a harmonic function $H$ in $V$ such that 
\begin{equation}
\h =|H|, \ \text{in} \ V \ \text{or}  \  \h =-|H|, \ \text{in} \ V 
\end{equation}
\begin{equation}
\supp \mu_{\lfloor V} =\{z \in V ; H(z)=0 \}.
\end{equation}
Furthermore we have that $\nabla H(z_0) \neq 0$ and the set $\{z \in V ;H(z)=0\}$ is a smooth simple curve diffeomorphic to a straight line.
\end{thm}

\begin{figure}[ht!]
\begin{center}
\scalebox{0.5}{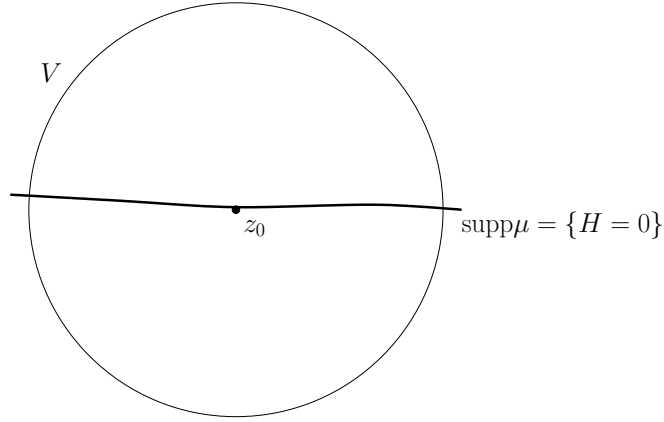}
\end{center}
\caption{Near a regular point $\supp \mu$ is a smooth curve.}
\end{figure}

Near a point $z_0$ such that $\omega_{\h}(z_0)=0$ the behavior of $\h$ and the geometry of the support of $\mu$ is a little bit more complicated. Nevertheless if $z_0$ is a zero of even order of $\omega_{\h}$ the situation is similar.

\begin{thm}\label{theorem2}
Let $ z_0 \in \supp\mu$, with $(\h,\mu)$ which satisfy assumptions \eqref{1}, \eqref{3}, \eqref{2}, and such that $z_0$ is a zero of even order of $\omega_{\h}$. We assume that $\h(z_0)=0$.  Then there exist a neighborhood $V$ of $z_0$, a harmonic function $H$ in $V$ and a function $\theta: V \rightarrow \{ \pm 1 \}$ such that 
\begin{equation}
\h(z) =\theta (z) H(z) \ \ \text{in} \ V.
\end{equation}
The function $\theta H$ is continuous and $\nabla H(z_0) =0$. Besides the support of $\mu_{\lfloor V}$ is a union of smooth curves included in $ \{z \in V ; H(z)=0 \}$ which end at $z_0$.
\end{thm}

A key ingredient in the proof of the previous theorem is the local structure of the set of zeros of harmonic functions (see \textit{e.g.} \cite{HartmanWintner} or \cite{fonctionharmonique}).
 
\begin{thm}[\cite{fonctionharmonique}]\label{courbes harmoniques}
Let $H$ be a harmonic function defined on an open set $D \subset \R²$. We let $Z_0(H):=\{z\in D ; H(z)=0\}$. Suppose $z_0\in D$, $H(z_0)=0$ and $H$ is not identically zero. Then there exist a unique integer $n=n(H,z_0)\geq 1$, a neighborhood $U(z_0)$ of $z_0$ in $D$ and $n$ analytic curves 
$$\gamma_k:]-1,1[ \rightarrow U(z_0), \ \ \ \ \ (k=1,2,...,n) $$
such that $\gamma_k(0)=z_0$ and: 
\begin{itemize}
\item[1)] $Z_0(H) \cap U(z_0)= \displaystyle{\cup_{k=1}^n \gamma_k}$ (where $\gamma_k$ denotes the set $\{\gamma_k(t) ; t \in]-1,1[\})$
\item[2)] $\ang(\gamma_k, \gamma_{k+1})=\frac{\pi}{n}$, $k=1,..n$, where $\gamma_{n+1}$ denotes $\gamma_1$ and $\ang(\gamma_k, \gamma_{k+1})$ is the angle between $\gamma_k$ and $\gamma_{k+1}$ at $z_0$.
\item[3)] There exists an analytic diffeomorphism $\phi: U(z_0) \rightarrow B(0,1)$ such that 
$$\phi \circ \gamma_k(t)=t \exp(i\theta_k)$$
where $t\in ]-1,1[$, $k=1,...,n, \theta_k=\frac{\pi}{2n}+\frac{(k-1)\pi}{n}$.
\end{itemize} 
This means that $\Gamma_k=\phi(\gamma_k)$ are $n$ symmetrically placed diameters of $B(0,1)$.
\end{thm}

%
%

\textbf{Remark:} Note that in Theorem \ref{theorem2} it can happen that the support of $\mu$ is strictly contained in the set $\{z \in v ;H(z)=0\}$. In this case we can not have $\h=|H|$. This is illustrated by the following example: we set $h(re^{i\varphi})= \theta (\varphi) r^2\cos(2\varphi)$, for $r\in [0,1]$, $\varphi \in [0,2\pi[$ and

$$\theta(\varphi)= \begin{cases} -1, \  \text{if}  \ \frac{\pi}{4}\leq \varphi \leq \frac{3\pi}{4} \\
+1, \ \text{otherwise}.
\end{cases}$$
This function $h$ satisfies \eqref{1}, \eqref{3}, \eqref{2}. In particular one can check that  $\Delta h=\mu$ with $\supp(\mu)=D_1 \cup D_2$ where $D_1=\{ z=re^{i\varphi}, 0\leq r \leq 1 \ \text{and} \ \varphi=\frac{\pi}{4}\}$, $D_2=\{ z=re^{i\varphi}, 0\leq r \leq 1 \ \text{and} \ \varphi=\frac{3\pi}{4}\}$\\

\begin{figure}[ht!]
\begin{center}
\scalebox{0.5}{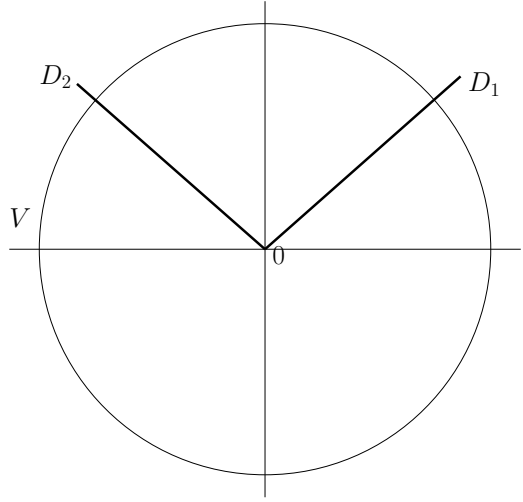}
\end{center}
\caption{An example of the geometry of $\supp \mu$ near a critical point of $\h$.}
\end{figure}

When $z_0 \in \supp(\mu)$  is a zero of odd order of $\omega_{\h}$ we must use multivalued harmonic function. 

\begin{thm}\label{theorem3}
Let $ z_0 \in \supp\mu$ with $(\h,\mu)$ which satisfy assumptions \eqref{1}, \eqref{3}, \eqref{2}, and such that $z_0$ is a zero of odd order of $\omega_{\h}$. We assume that $\h(z_0)=0$. Then there exist a neighborhood $V$ of $z_0$, a multivalued harmonic function $H_1$ in $V$ such that $H:=|H_1|$ is a single-valued function and a function $\theta: V \rightarrow \{ \pm 1 \}$ such that 

\begin{equation}
\h(x) =\theta(x) H(x) \ \ \text{in} \ V,
\end{equation}
the function $\theta H$ being continuous and $\nabla H (z_0)=0$. Besides the support of $\mu_{\lfloor V}$ is a union of smooth curves included in $ \{z \in V ; H(z)=0 \}$ which end at $z_0$.

Furthermore the function $H_1$ is such that: there exist an unique integer $n \geq 1$, a small number $r >0$ and a biholomorphism $\Phi :B(0,r) \rightarrow V$ such that $\Phi(0)=z_0$ and 
\begin{equation}
H_1 \circ \Phi (z)= \text{Re} (z^{n+\frac{1}{2}}), \ \ \text{for} \ z \in B(0,r)
\end{equation}
\end{thm}

Thanks to the property satisfied by the function $H_1$ in the previous theorem we can obtain a description of the set of zeros of $H_1$ similar to Theorem \ref{courbes harmoniques}.                                                                                                                                                                                                                 

\begin{thm}\label{courbesharmoniquesmulti}
let $H_1$ be as in the previous Theorem \ref{theorem3}. Then there exist $2n+1$ analytic curves 
$$\gamma_k:]-1,1[ \rightarrow V, \ \ (k=1,2,...,2n+1)$$
such that $\gamma_k(0)=z_0$ and
\begin{itemize}
\item[1)] $\{z \in \R^2 ; H_1(z)=0\} \cap V= \displaystyle{\cup_{k=1}^{2n+1} \gamma_k}$
\item[2)] $\ang(\gamma_k,\gamma_{k+1})=\frac{2\pi}{n+1}$, $k=1,...,2n+1,$ where $\gamma_{2n+n}$ denotes $\gamma_1$ and $\ang(\gamma_k,\gamma_{k+1})$ is the angle between $\gamma_k$ and $\gamma_{k+1}$ at $z_0$.
\item[3)] There exists an analytic diffeomorphism $\phi: V \rightarrow B(0,1)$ such that 
$$\phi \circ \gamma_k(t)=t \exp(i\theta_k)$$
where $t\in ]-1,1[$, $k=1,...,2n+1,$ and  $\theta_k=\frac{\pi}{2n+1}+\frac{2(k-1)\pi}{2n}$.
\end{itemize}
\end{thm}

\begin{figure}[ht!]
\begin{center}
\scalebox{0.5}{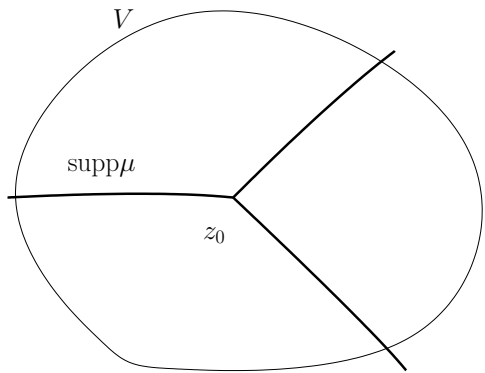}
\end{center}
\caption{Illustration of Theorem \ref{theorem3}.}
\end{figure}

In order to conclude this introduction we would like to comment on the hypothesis \eqref{1}, \eqref{3}, \eqref{2}.
First note that the fact that $\h$ is in $H^1$ (or equivalently that $\mu \in H^{-1}$) is essential to assume \eqref{2} since we take the divergence of the tensor $T_\mu$ in the sense of distributions we must have that its coefficients are in $L^1_{loc}$. Then we want to give an example which shows that \eqref{2} does not necessarily imply that $\mu$ is a Radon measure. The example is the following: one can take $h$ defined on $[0,1]$ such that $h(0)=0$ and 
$$h'(x)= \begin{cases} +1, \ \text{if} \ x \in ]\frac{1}{n+1},\frac{1}{n}[ \ \text{with} \ n \ \text{even} \\
-1, \  \text{if} \ x \in ]\frac{1}{n+1},\frac{1}{n}[ \ \text{with} \ n \ \text{odd}.
\end{cases} $$
We then have that $h \in H^1([0,1])$, and $h$ satisfies $\omega_h=|h'(x)|^2=1$ is holomorphic. But $\Delta h= \sum_{n=2}^{+\infty} \delta_{\frac{1}{n}}$ is not a Radon measure.\\

The paper is organized as follows: In Section 2 we explain the physical motivations for studying this problem.  Section 3 is devoted to the description of the measure $\mu$ near a point $z_0$ such that $\omega_{\h}(z_0) \neq 0$. In Section 4 we discuss the case of a zero of even order of $\omega_{\h}$ and in Section 5 the case of a zero of odd order of $\omega_{\h}$.

\section{Physical motivations of the problem}
\subsection{Connections to Ginzburg-Landau vortices without magnetic field.}
The conditions \eqref{1}, \eqref{3}, \eqref{2} are motivated by the problem of describing limiting vorticities for the critical points $(u_\e)_\e $ of the Ginzburg-Landau energy without magnetic field

\begin{equation}
E_\e(u)=\frac{1}{2}\int_\Omega |\nabla u|^2dx +\frac{1}{4\e^2} \int_\Omega (1-|u|^2)^2dx.
\end{equation}
\noindent Here $u$ is a complex-valued function called the \textit{order parameter} and its isolated zeros are called \textit{vortices}. The Ginzburg-Landau theory is a model for describing the superconductivity. The Ginzburg-Landau system without magnetic field was studied by Béthuel-Brézis-Hélein in \cite{BBH}. Later on Sandier-Serfaty in \cite{SandierSerfaty} studied the Ginzburg-Landau system with magnetic field which is a more physically relevant model. The vortices are important features of the model. They correspond to small regions in the superconducting sample where the superconductivity is destroyed. Let $\Omega$ be a bounded domain in $\R^2$. We consider a family $(u_\e)_{\e>0}$ of solutions of 
\begin{equation}\label{G.L equation}
-\Delta u_\e =\frac{u_\e}{\e^2}(1-|u_\e|^2) \ \ \text{in} \ \Omega.
\end{equation}
We assume that $|u_\e| \leq 1$ in $\Omega$ and 

\begin{equation}
E_\e(u_\e) < C_0 \e^{\alpha -1}, \ \alpha > \frac{2}{3}
\end{equation}
for every $\e>0$. We let $j_\e=\langle iu_\e, \nabla u_\e \rangle $ where $\langle .,.\rangle$ denotes the inner product in $\C$ identified with $\R^2$.  We also let $\mu_\e= \curl j_\e$. Here $j_\e$ describes superconducting currents and $\mu_\e$ is the vorticity of these currents. A direct calculation shows that $\dive j_\e=0$ hence we can write $j_\e= \nabla^\perp h_\e$ for some function $h_\e$.  Furthermore this function satisfies the following equation 

\begin{equation}\label{G.L measure}
\left\{
\begin{array}{rcll}
\Delta h_\e & = &\mu_\e \ \ \text{in} \ \Omega \\
\partial_\nu h_\e & =& \langle j_\e , \tau \rangle \ \ \text{on}  \ \partial \Omega.
\end{array}
\right.
\end{equation}

\noindent Here $\nu$ is the outward pointing normal to $\partial \Omega$ and $\tau= \nu^\perp$. By the solution to \eqref{G.L measure} we mean the solution with zero average in $\Omega$. We split $h_\e$ into two pieces: let us define $h_\e^0$ and $h_\e^1$ by
\begin{equation}
\left\{
\begin{array}{rcll}
-\Delta h_\e^1 &=&\mu_\e \ \ \text{in} \ \Omega, \\
h_\e^1 & =& 0 \ \ \text{on} \ \partial \Omega.
\end{array},
\ \ h_\e^0=h_\e-h_\e^1. \nonumber
\right.
\end{equation}

We recall the following result which describes the behavior of the vorticity measure as $\e$ goes to $0$ (see \cite{SanSer1} and \cite{SandierSerfaty}). 

\begin{thm}[Theorem 13.2 in \cite{SandierSerfaty}]\label{SanSer}
\begin{itemize}
\item[A)] Let $\{u_\e\}_{\e>0}$ be solutions of \eqref{G.L equation}. Then for any $\e>0$, there exists a measure $\nu_\e$ of the form $2\pi\sum_id_i^\e \delta_{a_i^\e}$ where the sum is finite, $a_i^\e \in \Omega$ and $d_i^\e \in \mathbb{Z}$ for every $i$, such that, letting $n_\e=\sum_i|d_i^\e|$,
\begin{equation}
n_\e \leq C \frac{E_\e(u_\e, \mathcal{B_\e})}{|\log\e|},
\end{equation}
where $\mathcal{B_\e}$ is a union of balls of total radius less than $C\e^{2/3}$, and such that 
\begin{equation}\label{13-17}
\|\mu_\e -\nu_\e\|_{W^{-1,p}(\Omega)}\|\mu_\e-\nu_\e\|_{(C^0(\Omega))^*} \rightarrow 0,
\end{equation}
for some $p\in (1,2)$.
\item[B)] Let $\{\nu_\e\}_\e$ be any measures of the form $2\pi\sum_id_i^\e \delta_{a_i^\e}$ satisfying \eqref{13-17}, let $n_\e=\sum_i |d_i^\e|$, and let $\{M_\e\}_\e$ be positive real numbers such that $\{h_\e^0/M_\e\}_\e$ converges in $L^1_{loc}(\Omega)$ to a function $H_0$. Then $H_0$ is harmonic and, possibly after extraction, one of the following holds.
\begin{itemize}
\item[0)] $n_\e=0$ for every $\e$ small enough and then $\mu_\e$ tends to $0$ in $W^{-1,p}(\Omega)$.
\item[1)] $n_\e=o(M_\e)$ is nonzero for $\e$ small enough, and then $\mu_\e/n_\e$ converges in $W^{-1,p}(\Omega)$ to a measure $\mu$ such that 
$$\mu \nabla H_0=0,$$
hence the support of $\mu$ is contained in the set of critical points of $H_0$.
\item[2)]$M_\e \sim \lambda n_\e$, with $\lambda>0$, and then $\mu_\e/M_\e$ converges in $W^{-1,p}(\Omega)$ to a measure $\mu$, and $h_\e/M_\e$ converges in $W^{1,p}_{loc}(\Omega)$ to a solution of $\Delta \h=\mu$ in $\Omega$. Moreover the symmetric $2$-tensor $T_\mu$ with coefficients $T_{ij}$ given by 
\begin{equation}\label{13-18}
T_{ij}=-\partial_i\h\partial_j\h+\frac{1}{2}|\nabla \h|^2\delta_{ij}
\end{equation}
is divergence-free in finite part (see Definition \ref{divergencefree} below).
\item[3)] $M_\e=0(n_\e)$, and then $\mu_\e/n_\e$ converges in $W^{-1,p}(\Omega)$ to a measure $\mu$, and $h_\e/n_\e$ converges in $W^{1,p}_{loc}(\Omega)$ to the solution of 
\begin{equation}
\left \{
\begin{array}{rcll}
\Delta \h &=&\mu \ \ \text{in} \ \Omega \\
\h&=&0 \ \ \text{on} \ \partial \Omega.

\end{array}
\right.
\end{equation}
Moreover the symmetric $2$-tensor $T_\mu$ with coefficients $T_{ij}$ given by \eqref{13-18} is divergence-free in finite part.
\end{itemize}
\end{itemize}
In cases 2) and 3), if $\mu \in H^{-1}(\Omega)$ then solutions of $\Delta \h= \mu$ are in $H^{1}_{loc}(\Omega)$. Thus $T_\mu$ is in $L^1_{loc}(\Omega)$ and we have that $\dive (T_\mu)=0$ in the sense of distributions. In other words $\h$ is stationary harmonic.
\end{thm}

Hence we can see that the limiting vorticity in cases 2), 3), with the additional hypothesis that $\mu \in H^{-1}(\Omega)$ satisfies condition \eqref{3}, \eqref{2}. Understanding the limiting measure $\mu$ will in turn give qualitative information on the behavior of vortices. 

\noindent We now recall the definition of the notion of divergence-free in finite part taken from \cite{SandierSerfaty}.

\begin{definition}\label{divergencefree} 
Assume $X$ is a vector field in $\Omega$. We say that $X$ is \textit{divergence-free in finite part} if there exists a family of sets $\{E_\delta\}_{\delta>0}$ such that
\begin{itemize}
\item[1.]For any compact $K \subset \Omega$, we have $\lim_{\delta \rightarrow 0} \capa_1(K\cap E_\delta)=0.$
\item[2.] For every $\delta >0$, $X \in L^1(\Omega \setminus E_\delta).$
\item[3.] For every $\zeta \in C^\infty_c(\Omega)$,
$$\int_{\Omega \setminus F_\delta} X\cdot \nabla \zeta =0$$
where $F_\delta=\zeta^{-1}(\zeta(E_\delta)).$
If $T$ is a 2-tensor with coefficients $\{ T_{ij} \}_{1\leq i,j\leq 2}$, we say that $T$ is divergence free in finite part if the vectors $T_i=(T_{i1},T_{i2})$ are, for $i=1,2$.
\end{itemize}
\end{definition}
\noindent In this definition we denoted by $\capa_1$ the $1$-capacity of a set $E \subset \R^2$ and we recall from Evans-Gariepy \cite{EvansGariepy} that the $p$-capacity  ($1\leq p <2$) of a set $E$ is defined as 

$$\capa_p(E)= \inf\{\int_{\R^2} |\nabla \varphi |^p; \varphi \in L^{p^*}(\R^2), \nabla \varphi \in L^p(\R^2), A \subset int(\varphi \geq 1) \},$$
where $int(A)$ denotes the interior of $A$ and $p^*=\frac{2p}{2-p}$.
We would like mention that in \cite{NamLe}, the author studied limiting vorticity measures associated to the Ginzburg-Landau system with magnetic field. This leads to conditions analog to \eqref{1}, \eqref{3}, \eqref{2}. He investigated these conditions under the additional assumption that the measure $\mu$ is supported by a simple smooth curve. He then proved, among other things, that in that in this case $\mu$ has a fixed sign.
\subsection{Connections to the Euler System}
It turns out that conditions \eqref{3}, \eqref{2} are also related to the Euler equations for incompressible flow in fluid mechanics. They can be written as follows:

\begin{equation}\label{Euler}
\left \{
\begin{array}{rclll}
\partial_t v+ (v\cdot \nabla) v +\nabla p &=&0 \ \text{in} \ \Omega  \\
\dive(v)&=&0 \ \text{in} \ \Omega
\end{array}
\right.
\end{equation}
where $\Omega$ is an open set of $\R^2$. In this system $p$ is called the pressure and it is an unknown of the system. Here $v\cdot \nabla v :=v_1\partial_x v + v_2\partial_y v$, and $v$ is the velocity of the fluid. The system is stationary if it does not involve in time, \textit{i.e.}, if  $\partial_tv =0$ in $\Omega$. A quantity of particular interest in fluid mechanics is the vorticity of the fluid defined by 
\begin{equation}
\mu=\curl v. 
 \end{equation}
 We must be more specific to define the notion of solutions of the Euler system. Indeed we want to give a meaning to \eqref{Euler} for vector-fields which are only in $L^2(\Omega)$. First note that thanks to the condition $\dive(v)=0$ we can rewrite the stationary Euler system in the following form:

\begin{equation}\label{Euler2}
\left\{
\begin{array}{rcll}
\dive(v\otimes v) +\nabla p &=& 0 \\
\dive(v) &=& 0.
\end{array}
\right.
\end{equation}
where $(v\otimes v)$ is a $2\times 2$ matrix given by $(v\otimes v)_{i,j}= v_iv_j$, for $1 \geq i,j \geq 2$. The divergence of a matrix is the sum of the divergence of the row. Let us denote by $\langle A,  B \rangle := \tra (A^t B)$ the inner product between two matrices.

\begin{definition}
Let $\Omega$ be an open set in $\R²$. We say that  $v \in L^2(\Omega,\R^2)$ is a \textit{weak solution} of \eqref{Euler2} if there exists $p \in L^1(\Omega)$ such that
\begin{equation}
\int_\Omega \langle v \otimes v, D\varphi \rangle +\int_\Omega p\dive(\varphi) =0 ,\ \ \forall \ \varphi \in\mathcal{C}_c^\infty(\Omega,\R^2). 
\end{equation}
\end{definition}

\begin{proposition}\label{linkwithEuler}
Let $\h$ satisfy \eqref{1}, \eqref{3}, \eqref{2}. We set $v= \nabla^\perp \h$. Then $v$ is a weak solution of the stationary Euler system with vorticity equal to $\mu$.
\end{proposition}

\begin{proof}
Let $\varphi=(\varphi_1,\varphi_2) \in \mathcal{C}_c^\infty (\Omega,\R^2)$, recall that $\nabla ^\perp h=(-\partial_yh,\partial_xh)$.
\begin{eqnarray}
\int_\Omega \langle \nabla^\perp h \otimes \nabla ^\perp h, D\varphi \rangle = \int_\Omega (\partial_yh)^2\partial_x\varphi_1-(\partial_yh\partial_xh)\left[\partial_y\varphi_1 +
\partial_x\varphi_2\right]+(\partial_xh)^2\partial_y\varphi_2. \nonumber
\end{eqnarray}
However because of the condition \eqref{2} we have 
\begin{eqnarray}
\int_\Omega \frac{1}{2}(\partial_yh^2-\partial_xh^2)\partial_x\varphi_1-(\partial_xh\partial_yh)\partial_y\varphi_1&=&0 \nonumber \\ 
\int_\Omega (-\partial_xh\partial_yh)\partial_x\varphi_2+\frac{1}{2}(\partial_xh^2-\partial_yh^2)\partial_y\varphi_2&=&0 .\nonumber 
\end{eqnarray}

\noindent Hence we can rewrite 

\begin{eqnarray}
\int_\Omega \langle \nabla^\perp h \otimes \nabla ^\perp h, D\varphi \rangle =\int_\Omega \frac{1}{2}(\partial_xh^2+\partial_yh^2)\partial_x\varphi_1 +\int_\Omega \frac{1}{2}(\partial_x h^2+\partial_yh^2)\partial_y\varphi_2. \nonumber
\end{eqnarray}
We then set $p=\frac{1}{2}|\nabla h|^2 \in L^1(\Omega)$ and we obtain that for all $\varphi \in \mathcal{C}_c^\infty(\Omega,\R^2)$ we have 

\begin{equation}
\int_\Omega \langle \nabla^\perp h \otimes \nabla ^\perp h, D\varphi \rangle=-\int_\Omega p\dive(\varphi). \nonumber
\end{equation}
Thus $v=\nabla^\perp h$ is a weak solution of stationary Euler system, with pressure $p=\frac{1}{2}|\nabla h|^2$ and with vorticity equal to $\curl \nabla ^\perp h= \Delta h =\mu$.
\end{proof}




The previous Proposition \ref{linkwithEuler} combined with Theorems \ref{theorem1}, \ref{theorem2}, 
\ref{theorem3}, implies that if $v$ is a weak solution of the Euler system \eqref{Euler2} such that 
$v=\nabla^\perp h$ with $h$ which satisfies hypothesis \eqref{1}, \eqref{3}, \eqref{2} then $v$ is a 
\textit{vortex sheet} solution of \eqref{Euler2}. The vortex sheet problem consists in finding a solution $(v,p)$ 
of \eqref{Euler} such that the initial data $v_{|t=0}=v_0$ satisfies that $\dive(v_0)=0$ and $\omega_0=\curl v_0= 
\delta_{\Sigma}$ with $\Sigma$ a compact smooth curve in $\R^2$. The existence of global solution of vortex sheet 
solutions of the Euler equation is due to J.M Delort in \cite{JMDelort}. Note that in his paper an important assumption for the proof of the existence of global solution of vortex sheet solutions is that the initial data $\omega_0=\curl v_0=\mu$ is a positive (or negative) measure. However in cases of theorems \ref{theorem2}, \ref{theorem3}, it can happen that $\mu$ has no sign. For example setting $h(r,\varphi)=\theta(\varphi)r^2\cos(2\varphi)$ with $\theta(\varphi)=+1$, if $-\frac{3\pi}{4}\leq \varphi\leq \frac{\pi}{4}$ and $\theta(\varphi)=-1$ if $\varphi \in [-\pi,\pi] \setminus [-\frac{3\pi}{4},\frac{\pi}{4}]$. Then one can check that $\Delta h$ is a measure with no fixed sign.\\

Let us mention that not all stationary solutions of the Euler system \eqref{Euler2} can be written as $v=\nabla^\perp h$ with $h$ which satisfies \eqref{1},\eqref{3},\eqref{2}. For example we take $v=(-y,x)$ for $x,y \in B(0,1)$. Then we can check that $v$ is a solution of \eqref{Euler2} with $p=\frac{1}{2}(x^2+y^2)$. We can write $v=\nabla^\perp h$ with $h=\frac{1}{2}(x^2+y^2)$. But $h$ satisfies that $\omega_h= (x-iy)^2$ and it is not holomorphic. Hence \eqref{2} is not satisfied. Note that in this case $\Delta h =1$ in $B(0,1)$. Such a solution is called a \textit{vortex patch}.

\subsection{Connections to system of point vortices}

A system of $N$-point vortices in evolution is described by the following system of ordinary differential equations 
\begin{equation}
\frac{dz_i}{dt}(t)=\nabla^ \perp \left[ \sum_{j=1,j\neq i}^N d_j \ln |z-z_j(t)| \right] (z_i(t)), \ \forall i=1,...,N.
\end{equation}
with $\nabla^\perp =(-\partial_y,\partial_x)$ and $d_j \in \mathbb{N}$. The points $z_i(t)$ are called vortices and $d_i$ are the \textit{degrees} of vortices. The system is \textit{stationary} if the vortices do not evolve in time, one then has
\begin{equation}\label{syspointvortices}
\sum_{j=1,j\neq i}^N d_j\frac{z_i-z_j}{|z_i-z_j|^2}=0, \ \ \ \forall i=1,...,N.
\end{equation}

A natural question is the following: \textit{What are the limiting vorticities of a stationary system of point vortices when the number of points tends to infinity?}

Let us reformulate precisely this question. Let $\Omega$ be a bounded domain, we are interested in Radon measure $\mu$ which satisfies the following conditions:

\begin{equation}\label{9}
\forall \e >0, \exists  N^\e \in \mathbb{N}, \ (z^\e_i)_{1\leq  i \leq N_\e} \in \Omega, \ d^\e_i \in \mathbb{Z} \ \text{ s.t.} \|\mu -2\pi  \sum_{i=1}^{N^{\e}} d^\e_i \delta _{z_i^\e} \|_{(\mathcal{C}^0(\Omega))^*} <\e 
\end{equation}
\begin{equation}\label{10}
\ \big((z^\e_i)_{1\leq  i \leq N_\e}, (d^\e_i)_{1\leq i \leq N^\e} \big) \text{ define a stationary system of point vortices.}
\end{equation}

The limiting vorticities of a stationary system of point vortices are described by a result analog to Theorem \ref{SanSer}:

\begin{thm}\label{pointvortex}
Let $\Omega$ be a bounded domain. Let $\mu$ be a Radon measure in $\Omega$ which satisfies \eqref{9} and \eqref{10}. There exists a function $u \in L^1_{loc}(\Omega)$ such that 
\begin{itemize}
\item[1)] $\Delta u =\mu$ \\
\item[2)] The tensor $$T_u=  \begin{pmatrix}
\frac{1}{2}\left[(\partial_y u) ^2-(\partial_x u) ^2\right] & -\partial_x u \partial_y u \\
-\partial_x u \partial_y u & \frac{1}{2}\left[(\partial_x u) ^2-(\partial_y u)^2\right]
\end{pmatrix} $$ is divergence-free in finite parts. 
Furthermore if $\mu$ is in $H^{-1}(\Omega)$ then $u$ is in $H^1(\Omega)$ and $\dive(T_u)=0$ in the sense of distributions. That is $u$ satisfies the conditions \eqref{1}, \eqref{3}, \eqref{2}.
\end{itemize}
\end{thm} 

Thanks to the previous theorem we see that studying the conditions \eqref{1}, \eqref{3}, \eqref{2} can be useful to obtain information  about the vorticity of a stationary system of point vortices when the number of vortices tends to infinity. The rest of this subsection is devoted to the definitions needed in the statement of Theorem \ref{pointvortex} and its proof. The definitions and some results are taken from \cite{SandierSerfaty} Chapter 13.

%
In this section we use an equivalent definition of divergence-free in finite part:
 
\begin{definition}\label{second}
Let $X$ be a vector field in $\Omega$, and $z_1,...z_N$ in $\Omega$ such that $X \in C^0(\Omega \setminus\{z_1,...,z_N\})$. We say that $X$ is divergence free in finite part if 
\begin{itemize}
\item[1.] $\dive(X) =0$ in $\mathcal{D}'(\Omega \setminus\{z_1,...,z_N\})$.
\item[2.] $\int_{\partial B(z_i,\delta)} X.\nu_i =0$, $\forall \ i=1,...,N, \ \forall \delta >0$ where $\nu$ denotes the outward unit normal to $\partial B(z_i,\delta)$.
\end{itemize}
\end{definition}

\noindent The equivalence between the two previous definitions can be proved using the coarea formula.
\begin{definition}
We say that $u$ is \textit{weakly stationary harmonic} if $T_u$ is divergence free in finite part.
\end{definition}

\begin{exa} $u(z)= \ln |z|$ is weakly stationary harmonic in $\R^2$.
\end{exa}
\begin{proof}
Let $z=x+iy$. We have that $\ln |z|$ is harmonic in $\R^2\setminus\{0\}$ and smooth in $\R^2\setminus \{0\}$. Then it is stationary harmonic in  $\R^2\setminus\{0\}$, that is $\dive(T_u)=0$ in $\R^2\setminus\{0\}$, with $T_u=\begin{pmatrix}
\partial_xu^2-\partial_yu^2 & 2\partial_xu \partial_y u \\
2\partial_xu \partial_y u & \partial_yu^2-\partial_xu^2
\end{pmatrix}$.
We want to show the second condition in the previous definition. Let $\delta>0$ we have 
$$\partial_xu^2-\partial_yu^2 =\frac{x^2-y^2}{|z|^2} \text{ and } 2\partial_xu \partial_y u =\frac{xy}{|z|^2}.$$
The outward unit normal to $\partial B(0,\delta)$ is $\nu=\frac{z}{|z|}$. Hence, for all $\delta>0$ small:
\begin{eqnarray}
\int_{\partial B(0,\delta)}(\partial_xu^2-\partial_yu^2)\nu_1+ (2\partial_xu \partial_y u)  \nu_2& 
=&\int_{\partial B(0,\delta)} \frac{x (x^2 + y^2)}{ |z| ^3} \nonumber \\
&=& \delta \int_0^{2\pi} \cos(\varphi)d\varphi=0. \nonumber 
\end{eqnarray}

\noindent The integral of the other component of $T_h$ is computed the same way and we also find that it is equal to $0$. Thus $u$ is weaky stationary harmonic.
\end{proof}

We can associate to a system of point vortices \eqref{syspointvortices} the measure $ \sum_{i=1}^N d_i \delta_{z_i}$, where we denoted by $\delta_{z_0}$ the Dirac mass in $z_0$. Let us consider the particular solution of 
\begin{equation}
\Delta u=\frac{2\pi}{M_N} \sum_{i=1}^N d_i \delta_{z_i}
\end{equation}
given by 
\begin{equation}
u(z)= \frac{1}{M_N} \sum_{i=1}^N d_i \ln |z-z_i|, \ \text{where} \ M_N= \sum_{i=1}^N |d_i|.
\end{equation}

\begin{proposition}
The points $(z_i)_{1\leq i \leq N} \in \R^2$ form a stationary system of point vortices if and only if the function 
$u(z)= \frac{1}{M_N} \sum_{i=1}^N d_i \ln |z-z_i|$ is weakly stationary harmonic.
\end{proposition}

\textbf{Remark:} Note that if $u$ is not in $H^1(\Omega)$ then it does not make sense to say that $u$ is stationary harmonic that is why we need the notion of weak stationary harmonicity. 

\begin{proof}
We use Definition \ref{second}. Again away from the points $z_1,...,z_N$, $u$ is harmonic and smooth. Thus it is stationary harmonic. Near $z_1$ we have $$u(z)=\alpha_1 \ln|z-z_1|+H_1(z)$$ where $H_1(z):=\frac{1}{M_N}\sum_{i=2}^N d_i\ln|z-z_i|$ is harmonic near $z_1$ (in a neighborhood of $z_1$ which contains only $z_1$ and no other $z_i$) and $\alpha_1$ is a constant. Without loss of generality we can assume that $\alpha_1=1$ and $z_1=0$. We then have:
\begin{eqnarray}
\partial_xu^2-\partial_yu^2&=& \left(\frac{x}{|z|}+\partial_xH_1(z)\right)^2-\left(\frac{y}{|z|}+\partial_yH_1(z)\right)^2 \nonumber \\
&=&\frac{x^2-y^2}{|z|^2}+(\partial_xH_1 ^2 -\partial_yH_1^2)+2\left(\frac{x}{|z|}\partial_xH_1-\frac{y}{|z|}\partial_yH_1\right) \nonumber
\end{eqnarray}
\begin{eqnarray}
2\partial_x u \partial_y u&=&2\frac{xy}{|z|^2}+2\partial_xH_1\partial_yH_1+2\left(\frac{x}{|z|}\partial_yH_1 +\frac{y}{|z|}\partial_xH_1\right) \nonumber 
\end{eqnarray}

Thus
\begin{eqnarray}
\int_{\partial B(0,\delta)} (\partial_xu^2-\partial_yu^2)\nu_1+ (2\partial_xu \partial_y u)  \nu_2 =
\int_{\partial B(0,\delta)} \frac{x^2-y^2}{|z|^2}\nu_1+2\frac{xy}{|z|^2}\nu_2 \nonumber \\
+\int_{\partial B(0,\delta)}(\partial_xH_1 ^2 -\partial_yH_1 ^2)\nu_1+(2\partial_xH_1\partial_yH_1)\nu_2 \nonumber \\
+ \int_{\partial B(0,\delta)}2\left(\frac{x}{|z|}\partial_xH_1-\frac{y}{|z|}\partial_yH_1\right)\nu_1+2\left(\frac{x}{|z|}\partial_yH_1 +\frac{y}{|z|}\partial_xH_1\right) \nu_2. \nonumber
\end{eqnarray}

The first term in this sum is zero because $\ln|z|$ is weakly stationary harmonic. The second term is also zero because $H$ is harmonic, smooth, and hence stationary harmonic and weakly stationary harmonic. For the third term we can  use the fact that the normal on $\partial B(0,\delta)$ is $\nu=\frac{z-z_1}{|z-z_1|}$ to prove that it is equal to 
$2\int_{\partial B(0,\delta)} \partial_xH_1$. \\

\noindent Hence if $u$ is weakly stationary harmonic this term must be equal to zero for all $\delta$. Then dividing this quantity by $\delta$ and letting $\delta$ go to $0$ we find that $\partial_xH_1(z_1)=0$. With the same method applied to the other component of $T_u$ we obtain 
$$\int_{\partial B(0,\delta)} (2\partial_xu \partial_yu) \nu_1+ (\partial_yu^2-\partial_xu^2)\nu_2 =2\int_{\partial B(0,\delta)} \partial_yH_1.$$ Thus if $u$ is weakly stationary harmonic we find that $\nabla H_1(0)=0$. By repeating this argument near each $z_i$, we obtain that if $u$ is weakly stationary harmonic then $z_1,...,z_N$ form a stationary system of point vortices:
$$\sum_{j=1,j\neq i}^N d_j\frac{z_i-z_j}{|z_i-z_j|^2}=0 \ \forall i=1,...,N.$$ 
\end{proof}

We now prove Theorem \ref{pointvortex}. Let $\mu$ be a Radon measure which satisfies \eqref{9}, \eqref{10}. We set

\begin{equation}\label{limit}
u_{N^\e}:=\frac{1}{M_{N^\e}} \sum_{i=1}^N d_i^\e\ln|z-z^\e_i|
\end{equation}

\noindent with $M_{N^\e}=\sum_{i=1}^{N^\e}|d_i^\e|$. We want to prove that $u_{N^\e}$ converges to a function $u$ when $\e$ goes to $0$ such that $u$ satisfies $\Delta u= \mu$ and $u$ is weakly stationary harmonic. However we need to have a notion of convergence which preserves the notion of weak stationary harmonicity. This is the object of the following definition.

\begin{definition}[\cite{SandierSerfaty}]
We say (with some abuse of notation) that a sequence $(X_n)_n$ in $L^1(\Omega)$ \textit{converges in $L^1_\delta(\Omega)$} to $X$ if $X_n \rightarrow X$ in $L^1_{loc}(\Omega)$ except on a set of arbitrarily small 1-capacity, or precisely if there exists a family of sets $(E_\delta)_{\delta>0}$ such that for any compact $K \subset \Omega,$
\begin{equation}
\lim_{\delta \rightarrow 0} \capa_1(K \cap E_\delta)=0, \ \ \text{and} \ \forall \delta>0 \ \lim_{n\rightarrow +\infty} \int_{K \setminus E_\delta}|X_n-X|=0.
\end{equation}

\noindent We define similarly the convergence in $L^2_\delta$ by replacing $L^1$ by $L^2$ in the above.
\end{definition}

\begin{proposition}[\cite{SandierSerfaty}]
Assume $(X_n)_{n\in \mathbb{N}}$ is a sequence of divergence-free in finite part vector fields which converges to $X$ in $L^1_{\delta}(\Omega)$. Then $X$ is divergence free in finite part.
\end{proposition}

\begin{corollary}
Assume that $u_N$  is a sequence of weakly stationary harmonic functions such that $u_N$ converges to $u$ in $L^2_\delta(\Omega)$ and $\nabla u_N$ converges in $L^2_\delta(\Omega)$ then $u$ is weakly stationary harmonic.
\end{corollary}

Thus to prove Theorem \ref{pointvortex} we only need to prove that the functions $u_{N^\e}$ converge in $L^2_\delta(\Omega)$ to a function $u$ such that $\nabla u_{N^\e}$ converge to $\nabla u$ in $L^2_\delta(\Omega)$. 

\begin{proposition}
Let $\Omega$ be a bounded open set in $\R²$. Let $\mu$ be a Radon measure in $\Omega$ such that \eqref{9},\eqref{10} hold. Let $u_{N^\e}=\frac{1}{M_{N^\e}} \sum_{i=1}^N d_i^\e\ln|z-z^\e_i|$ then there exists $u$ such that 
$u_{N^\e}$ converge in $L^2_\delta(\Omega)$ to  $u$ and $\nabla u_{N^\e}$ converge to $\nabla u$ in $L^2_\delta(\Omega)$.
\end{proposition}

\begin{proof}
We let 
\begin{equation}\label{mu_N}
\mu_{N^\e}:= \frac{2\pi}{M_{N^\e}}\sum_{i=1}^{N^\e} \delta_{z_i^\e}.
\end{equation}
Since $\Omega$ is bounded the measure $\mu_{N^\e}$ has compact support and  we can then write 

\begin{equation}\label{u_N}
u_{N^\e}=\ln|z|\ast \mu_{N^\e}
\end{equation}
\noindent where $\ast$ denotes the convolution product. 
Then for all $\varphi$ in $\mathcal{C}^\infty _c(\R^2)$ we have 
\begin{eqnarray}
\langle u_{N^\e},\varphi\rangle = \langle \ln|z| \ast \mu_{N^\e}, \varphi \rangle 
= \langle \mu_{N^\e},\ln|z| \ast \varphi\rangle. \nonumber
\end{eqnarray}
Now we let $\e$ go to $0$, by hypothesis $\mu_{N^\e}$ converges to $\mu$ in $(\mathcal{C}^0(\Omega))^*$. Hence 
$$\langle u_{N^\e},\varphi \rangle \rightarrow \langle \mu, \ln|x| \ast \varphi \rangle. $$

\noindent This proves that $u_{N^\e}$ converges to some $u$ in the sense of distributions. \\

In the rest of the proof we drop the subscript $\e$ and consider the limit $N\rightarrow +\infty$ (if $N_\e$ stays bounded the proof is immediate). We follow closely the proof of Proposition 13.2 in \cite{SandierSerfaty}. We choose a bounded open set $\Omega'$ such that $\Omega \subset \subset \Omega'$. We can define $\mu_N$, $\mu$, $u_N$ and $u$ in $\Omega'$ (using formulas \eqref{u_N}, \eqref{mu_N} for $u_N$ and $\mu_N$ valid in $\R^2$ and passing to the limit in $\Omega'$). In $\Omega'$ we set 
\begin{equation} v_N=u_N-u, \ \ \ \ \alpha_N=\mu_N-\mu.
\end{equation}
We then have 
\begin{equation}\label{laplace=mesure}
\Delta v_N =\alpha_N \ \ \text{in} \ \Omega'.
\end{equation} 

\noindent It holds that $\displaystyle{\lim_{N \rightarrow +\infty} \|\alpha_N \|_{\mathcal{C}^0(\Omega')^*} =0}$. But since we have $W^{1,q}(\Omega') \hookrightarrow \mathcal{C}^0(\Omega')$ for $q>2$ we also have $\mathcal{C}^0(\Omega')^* \hookrightarrow W^{-1,p}(\Omega')$ for $p<2$. Thus we obtain $$\displaystyle{\lim_{N \rightarrow +\infty} \|\alpha_N \|_{W^{-1,p}(\Omega')}}=0 \ \text{for} \ p<2.$$ Now we let 

\begin{equation}\label{delta et F}
\delta_N= \left( \frac{\|\alpha_N \|_{W^{-1,p}(\Omega')}}{\|\alpha_N\|_{\mathcal{C}^0(\Omega')^*}+1} \right) ^{1/2} , \ \ F_N=\{x\in \Omega ; |v_N| \geq \delta_N \}.
\end{equation}

\noindent We have the following bound on the $p$-capacity of $F_N$ (\textit{cf.} \cite{EvansGariepy} p.158)
\begin{equation}\label{capacity}
\capa_p(F_N) \leq C \frac{\|v_n\|_{W^{1,p}(\Omega)}^p}{\delta_N^p}.
\end{equation}

\noindent We note note that by elliptic regularity theory $\|v_N\|_{W^{1,p}(\Omega)} \leq C \|\alpha_N\|_{W^{-1,p}(\Omega')}$ because of \eqref{laplace=mesure} and because $\Omega \subset \subset \Omega'$. Thus from \eqref{delta et F} and \eqref{capacity} we find that
$$\capa_p(F_N) \leq C \|\alpha_N \|_{W^{-1,p}(\Omega')}^{p/2}(\|\alpha_N\|_{\mathcal{C}^0(\Omega')^*}+1)^{p/2},$$

\noindent and therefore tends to 0 as $N$ goes to infinity. This implies in turn that $$\displaystyle{\lim_{N \rightarrow +\infty} \capa_1(F_N)=0}.$$ Now we use a cut-off function $\varphi \in \mathcal{C}^\infty(\overline{\Omega'})$ such that $|\varphi(x)|\leq 1$ for all $x\in \Omega'$, $\varphi \equiv 1$ in $\Omega$ and $\varphi=0$ on $\partial \Omega'$. We also set 
$$\tilde{F_N}=\{x\in \Omega' ; |\varphi v_N| \geq \delta_N \}. $$
We have that $F_N \subset \tilde{F_N}$ and $\tilde{F_N}\cap \Omega =F_N$ since $\varphi \equiv 1$ in $\Omega$. We  use the following truncated function:

\begin{equation}
\overline{\varphi v_N}= \left\{ \begin{array}{rcll}
\varphi v_N & \text{if} & |\varphi v_N| \leq \delta_N, \\
\delta_N & \text{if} & |\varphi v_N| >\delta_N.
\end{array}
\right.
\end{equation}

\noindent From a property of Sobolev functions (see \textit{e.g.} Lemma 7.7 in \cite{Gilbarg}), we have $\nabla (\overline {\varphi v_N})=0$ almost everywhere in $\tilde{F_N}$. We thus obtain:

\begin{eqnarray}
\int_{\Omega \setminus F_N} |\nabla v_N|^2 & \leq & \int_{\Omega' \setminus \tilde{F_N}}|\nabla (\varphi v_N )|^2 \nonumber \\
& \leq & \int_{\Omega'} \nabla (\varphi v_N) \cdot \nabla (\overline{\varphi v_N)} \nonumber \\
& \leq & \int_{\Omega'} -\Delta (\varphi v_N) \overline{\varphi v_N}. \nonumber 
\end{eqnarray}
The last inequality being true since $\varphi=0$ on $\partial \Omega'$. Using the Leibniz formula we obtain that 
$$ \Delta (\varphi v_N)= \Delta \varphi v_N +2 \nabla \varphi \cdot \nabla v_N +\varphi \Delta v_N. $$
Hence
\begin{eqnarray}
\int_{\Omega \setminus F_N} |\nabla v_N|^2 & \leq & \int_{\Omega '} |\Delta \varphi v_N (\overline{\varphi v_N})| +\int_{\Omega'} 2|\nabla \varphi | |\nabla v_N| |\overline{\varphi v_N}| +\int_{\Omega'} |\overline{\varphi v_N}|d\alpha_N \nonumber
\end{eqnarray}
where we used the fact that $\Delta v_N =\alpha_N$ in $\Omega'$. Now we use H\"older inequality to obtain that

\begin{eqnarray}
\int_{\Omega \setminus F_N} |\nabla v_N|^2 & \leq & C \delta_N \left(\int_{\Omega'} |v_N|^p\right)^{1/p} +C \delta_N \left(\int_{\Omega'} |\nabla v_N|^p\right)^{1/p} +\delta_N  \|\alpha_N\|_{\mathcal{C}^0(\Omega)^*} \nonumber \\
& \leq & C\delta_N \left(\|v_N\|_{W^{1,p}(\Omega')}+\|\alpha_N \|_{\mathcal{C}^0(\Omega')^*}\right). \nonumber
\end{eqnarray}
Thus 
\begin{equation}\label{13.23}
\lim_{N \rightarrow +\infty} \|\nabla v_N \|_{L^2(\Omega \setminus F_N)} =0.
\end{equation}
 We can also see, from the definition of $F_N$ and because $\Omega$ is bounded that 
 \begin{equation}\label{13.24}
\lim_{N \rightarrow +\infty} \| v_N \|_{L^2(\Omega \setminus F_N)} =0.
\end{equation}
 
 We conclude as in \cite{SandierSerfaty}. Since $\lim_{n \rightarrow +\infty} \capa _1 (F_N)=0$, there is a subsequence, still denoted by $\{n\}$, such that $\sum_n \capa _1(F_N) < +\infty$. We define 
 
 $$E_\delta = \bigcup _{N > \frac{1}{\delta}} F_N. $$

Then $\capa _1(E_\delta)$ tends to zero as $\delta$ goes to zero since it is bounded above by the tail of a convergent series. Moreover, for any $\delta >0$ we have $F_N \subset E_\delta$ when $N$ is large enough and therefore \eqref{13.23} and \eqref{13.24} imply that 

$$\lim_{N \rightarrow +\infty } \| v_N \|_{L^2(\Omega \setminus E_\delta)}=\lim_{N \rightarrow +\infty } \| \nabla v_N \|_{L^2(\Omega \setminus E_\delta)}=0.$$
\end{proof}

This proposition proves point 1) and 2) of Theorem \ref{pointvortex}. The next proposition shows that if we add the hypothesis that $\mu$ is in $H^{-1}(\Omega)$, then $u$ weakly stationary harmonic implies $u$ stationary harmonic. 

\begin{proposition}[Proposition 13.1 in \cite{SandierSerfaty}]
Assume that $X$ is divergence-free in finite part in $\Omega$ and that $X$ is in $L^1(\Omega \setminus E)$. Then for every $\zeta\in C^{\infty}_c(\Omega)$,
$$\int_{\Omega\setminus F}X\cdot \nabla \zeta =0,$$
where $F=\zeta^-1(\zeta(E))$. In particular if $X$ is in $L^1(\Omega)$, then $F= \emptyset$ in the above and therefore $\dive X=0$ in $\mathcal{D}'(\Omega)$.
\end{proposition}

If $\mu$ is in $H^{-1}(\Omega)$ we have seen in the introduction that $u$ is in $H^1(\Omega)$ and $T_u$ is in $L^1(\Omega)$. Thanks to the previous proposition $u$ weakly stationary harmonic implies $u$ stationary harmonic.

\section{First case: local behavior near a point $z_0$ such that $\omega_{\h}(z_0)=(\partial_x\h-i\partial_y\h)^2(z_0) \neq0$}

Let us recall that we consider a couple $(\mu, h_\mu)$ which satisfies

\begin{equation}\label{C1}
\h \in H^1(\Omega)
\end{equation}

\begin{equation}\label{C2}
\Delta \h = \mu, \ \text{in} \ \Omega
\end{equation}
\noindent where $\mu$ is a Radon measure and
\begin{equation}\label{C3}
\omega_{\h}= (\partial_x \h)^2-(\partial_y \h)^2-2i\partial_x \h\partial_y \h \ \text{is  holomorphic  in} \ \Omega
\end{equation}
In this section we  drop the subscript $\mu$ when there is no possible confusion. We denote by $B_r= B(z_0,r)=\{z \in \C ;|z-z_0| <R\}$ the ball of center $z_0$ and of radius $r$. The starting point of the proof of Theorem \ref{theorem1} is the following:

\begin{lemma}\label{starting point}
Let $h$ which satisfies \eqref{C1}, \eqref{C2}, \eqref{C3}. Let $z_0\in \Omega$ such that $\omega_h(z_0)\neq0$. Then there exist $R>0$, a function $\theta: B_R \rightarrow \{ \pm 1\}$ and a harmonic function $H:B_R \rightarrow \R$ such that 
\begin{equation}\label{ImP}
\partial_xh(z) -i \partial_y h(z)= \theta(z)\left[\partial_xH(z)-i\partial_yH(z)\right], \ \ \forall \ z \in B_R.
\end{equation}
\end{lemma}

\begin{proof}

It holds that 

\begin{equation}\nonumber
\omega_h=(\partial _x h)^2-(\partial_y h)^2-2i\partial_xh\partial_yh=4(\partial_zh)^2=(\partial_xh-i\partial_yh)^2.
\end{equation}
\noindent Thus $(\partial_xh-i\partial_yh)^2$ is a holomorphic function in $\Omega$. If $z_0$ is such that $\omega_h(z_0) \neq0$ then $f:=(\partial_xh-i\partial_yh)^2$ satisfies that  $f $ is holomorphic in $\Omega$ and f$(z_0)\neq0$. This implies that in a neighborhood $U$ of $z_0$ where $f(z)$ does not vanish there exists a function $g:U \rightarrow \C$ such that  $g^2=f$ in $U$.
We can hence deduce that there exists $\theta:U \rightarrow\{ \pm 1 \}$ such that 
\begin{equation}\label{important}\nonumber
\partial_xh-i\partial_yh=\theta(z)g(z) \ \text{in} \ U.
\end{equation}

\noindent From now on we take $U=B(z_0,R)=:B_R$  for $R$ sufficiently small. We then set $H(z) := \text{Re} \int_{z_0}^z g(s)ds.$ This is well defined since $B_R$ is simply connected. The function $H$ satisfies the following properties:

\begin{itemize}
\item[1)] $H$ vanishes at $z_0$.
\item[2)] $H$ is harmonic in $B_R$ because it is the real part of an holomorphic function.
\item[3)] $2\partial _z H (z)= g(z)$ or equivalently $\partial_x H -i\partial_y H =g$
\end{itemize}
Thus 
\begin{equation}\label{importantbis}
\partial_xh-i\partial_yh=\theta (\partial_xH-i \partial_yH), \ \text{in} \ B_R.
\end{equation}
Besides we have that $\nabla H(z_0) \neq0$ since $|\nabla H (z_0)|^2= |\omega_h(z_0)|^2\neq0$.
\end{proof}
\noindent We set:
\begin{eqnarray}\label{+}
B_R^+ := \{z \in B_R ; \theta(z)=+1 \}  \ \ \
B_R^-:=\{z \in B_R ; \theta(z)=-1 \}. 
\end{eqnarray}

\textbf{Idea of the proof of Theorem \ref{theorem1} :} The strategy of the proof is the following: we first show that the function $\theta$ is in $BV(B_R)$. Hence $B_R^+$ and $B_R^-$ are sets of finite perimeter in $B_R$. It turns out that the support of $\mu_{\lfloor{B_R}}$ is equal to the essential boundary of $B_R^+$ minus the (topological) boundary $\partial B_R$. Then we  use a theorem of structure of sets of finite perimeter in $\R^2$ due to Ambrosio-Caselles-Morel-Masnou in \cite{ACMM} to decompose the essential boundary of $B_R^+$ as a disjoint union of Jordan curves. Because of the relation \eqref{ImP} we are able to show that these Jordan curves are unions of some part of the boundary $\partial B_R$ and of level curves of the harmonic function $H$. Since $\mu$ is a Radon measure we  prove that there can not be an infinite number of level curves of $H$ in the support of $\mu$ near $z_0$ (otherwise $\mu(B_R)=+\infty$). Then we can take a smaller open set $V$ containing $z_0$ such that the support of $\mu_{\lfloor V}$ is the set of zeros of $H$. In $V$ we can use the fact that $\nabla (h-\theta H)=0$ or use the maximum principle to obtain that $h=+|H|$ or $h=-|H|$.


\begin{lemma}\label{lemma1}
Let $h$ which satisfies \eqref{C1}, \eqref{C2}, \eqref{C3}. Let $R>0$ be small enough, $\theta :B_R \rightarrow \{ \pm1\}$ and $H:B_R \rightarrow \R$ such that \eqref{ImP} holds. Then $\theta$ is in $BV(B_R)$. 
\end{lemma}
\begin{proof}
We set $g=\partial_x H-i\partial_y H$. Since $H$ is harmonic it holds that $g$ is holomorphic. Since $z_0$ is not a zero of the function $f=(\partial_xh-i\partial_yh)^2$ we have that $g$ does not vanish in $B_R$. Then we can write
$$\theta(z)=\frac{\partial _x h(z)-i\partial_y h(z)}{g(z)}.$$
We obtain that $\theta$ is in $L^1(B_R)$ since $h$ is in $H^1(\Omega)$ and $g$ is in $\mathcal{C}^\infty (\overline{B_R})$.

\noindent Furthermore we can differentiate $\theta$ in the sense of distributions using the Leibniz rule since $g \in \mathcal{C}^\infty(B_R)$. We obtain 

$$\partial_x\theta =\frac{(\partial^2_{xx}h-i\partial^2_{yx}h)g-\partial_xg(\partial_xh-i\partial_y h)}{g^2}, \ \ \ \partial_y\theta =\frac{(\partial^2_{xy}h-i\partial^2_{yy}h)g-\partial_yg(\partial_xh-i\partial_y h)}{g^2}.$$
Summing these two equalities it comes

$$\partial_x \theta +i\partial_y\theta =\frac{\Delta h}{g}-\frac{(\partial_xg+i\partial_yg)(\partial_xh-i\partial_yh)}{g^2}.$$
But since $g$ is holomorphic in $B_R$ it holds that $\partial_{\bar{z}} g=\frac{1}{2}[\partial_xg+i\partial_yg ]=0.$ Hence $\partial_x \theta +i\partial_y\theta =\frac{\Delta h}{g}.$
Now $\Delta h=\mu$ is a Radon measure and we can write $\partial_x\theta =\text{Re}(\frac{1}{g})\Delta h$ , $\partial_y\theta =\text{Im}(\frac{1}{g})\Delta h$. Let us denote by $\langle.,. \rangle$ the duality bracket for distributions. For all $\varphi \in \mathcal{C}^1_c(B_R,\R^2)$ with $|\varphi| \leq 1$, we have 

\begin{eqnarray}
\int_{B_R}\theta \dive \varphi =-\langle \partial_x \theta,\varphi_1 \rangle -\langle \partial_y \theta,\varphi_2 \rangle = -\int_{B_R}\text{Re}(\frac{1}{g})\varphi_1 d\mu- \int_{B_R}\text{Im}(\frac{1}{g})\varphi_2 d\mu. \nonumber
\end{eqnarray}
Hence we obtain

\begin{equation}\nonumber
|\int_{B_R} \theta \dive \varphi | \leq \| \frac{1}{g}  \|_{L^\infty} \mu (B_R) < +\infty
\end{equation}

\noindent which means that $\theta$ is in $BV(B_R)$ by definition.
\end{proof}

Recall that a set $E \subset \Omega$ is a set of finite perimeter in $\Omega$ if its characteristic function $\chi _E$ is in $BV(\Omega)$. We have 
$$ \chi_{B_R^+} =\frac{1}{2}(1+\theta), \ \ \  \chi_{B_R^-} =\frac{1}{2}(1-\theta). $$
Hence $B_R^+$ and $B_R^-$ are sets of finite perimeter in $B_R$. We need several definitions and results from the theory of sets of finite perimeter we recall these notions now and we refer the reader to the books \cite{EvansGariepy}, or \cite{AmbrosioFuscoPallara} for the proof of these results.

\begin{thm}[\cite{EvansGariepy} p.167]
Let $E$ be a set of locally finite perimeter in $\Omega$, then there exists a Radon measure on $\Omega$ denoted by $\|\partial E \|$ and a $\|\partial E \|$-measurable function $\nu_E: \Omega \rightarrow \R$ such that
\begin{itemize}
\item[1)] $|\nu_E(x)|=1$ $\|\partial E \|$- a.e. , and 
\item[2)] $\int_E \dive \varphi dx =\int_\Omega \varphi \cdot \nu_E \ d \|\partial E \|$ for all $\varphi \in \mathcal{C}^1_c(\Omega,\R^n)$.
\end{itemize}
\end{thm}

We present two notions of ``boundary" of sets of finite perimeter:

\begin{definition}
Let $E$ be a set of locally finite perimeter in $\R^n$ and $x\in \R^n$. We say that $x\in \partial^\star E$, \textit{the reduced boundary} of $E$, if 
\begin{itemize}
\item[i)] $\|\partial E \| (B(x,r))>0$ for all $r>0$, 
\item[ii)] $\displaystyle{ \lim_{r\rightarrow 0 } \frac{1}{|B(x,r)|}\int _{B(x,r)} \nu_E d\|\partial E \|=\nu_E(x)}$, and
\item[iii)] $|\nu_E(x)|=1$.
\end{itemize}
\end{definition}

\begin{definition}
Let $E$ be a Lebesgue measurable set in $\R^n$ and  $x\in \R^n$. We say $x \in \partial_\star E$, \textit{the measure theoretic boundary or essential boundary} of $E$ if
$$\limsup_{r\rightarrow 0} \frac{|B(x,r) \cap E|}{r^n} >0 \ \text{ and }  \ \limsup_{r\rightarrow 0} \frac{|B(x,r) \setminus E|}{r^n} >0. $$
\noindent(Here $|A|$ denotes the n-Lebesgue measure of a set in $\R^n$).
\end{definition}

The structure of the reduced boundary of a set of locally finite perimeter in $\R^n$ is described by the following theorem:

\begin{thm}[\cite{EvansGariepy} p.205]
Assume $E$ has locally finite perimeter in $\R^n$.
\begin{itemize}
\item[i)]Then
$$\partial^\star E =\bigcup_{k=1}^\infty K_k \cup N ,$$
where 
$$\|\partial E\| (N)=0$$
and $K_k$ is a compact subset of a $\mathcal{C}^1$-hypersurface $S_k$ ($k=1,2,...$).
\item[ii)] Furthermore, $\nu_E|{S_k}$ is normal to $S_k$ ($k=1,...$) and
\item[iii)] $\|\partial E\|= \mathcal{H}^{n-1}_{ \lfloor _{\partial^\star E}}$.
\end{itemize}
\end{thm}

We have a relation between the reduced and the essential boundary.

\begin{proposition}
\begin{itemize}
\item[i)] $\partial^\star E \subset \partial_\star E$.
\item[ii)] $\mathcal{H}^{n-1}(\partial_\star E \setminus \partial^\star E)=0$.
\end{itemize}
\end{proposition}

We will also use the following theorem

\begin{thm}[Gauss-Green formula \cite{EvansGariepy} p.209] \label{Gauss-Green}
Let $E \subset \R^n$ have locally finite perimeter.
\begin{itemize}
\item[i)] Then $\mathcal{H}^{n-1}(\partial_\star E \cap K) < +\infty$ for each compact set $K \subset \R^n$.

\item[ii)] Furthermore, for $\mathcal{H}^{n-1}$ a.e. $x \in \partial_\star E$, there is a unique measure theoretic unit outer normal $\nu_E(x)$ such that
\begin{equation}\label{Gauss-Green2}
\int_E \dive(\varphi)dx=\int_{\partial_\star E} \varphi \cdot\nu_E d\mathcal{H}^{n-1}
\end{equation}
for all $\varphi \in \mathcal{C}^1_c(\R^n,\R^n)$.
\end{itemize}
\end{thm}

Since $B_R^+$ is a set of finite perimeter in $B_R$, we denote by $\nu_{B_R^+}$ its measure theoretic (or generalized) outer normal.
\begin{lemma}\label{colinear}
Let $h$ as in Lemma \ref{starting point}. Let $\theta$, $H$ given by Lemma \ref{starting point}. Let $B_R^+=\{z \in B_R; \theta(z)=+1\}$, thanks to Lemma \ref{lemma1}, $B_R^+$ is a set of finite perimeter in $B_R$ and we have:   the generalized normal $\nu_{B_R^+}$ is collinear to $\nabla h$ and $\nabla H$, $\mathcal{H}^{1}_ {\lfloor{\partial_\star B^+_R \setminus \partial B_R}}$ almost everywhere in $B_R$.
\end{lemma}

\begin{proof}
Let us recall that, because of \eqref{ImP} we have  
\begin{equation}\label{previous2}
\partial_x h -i\partial_y h=\theta (\partial_x H -i\partial_y H).
\end{equation}
  In the sense of distributions we have $\partial_x \partial_y h =\partial_y \partial_x h$. Thus we obtain $\partial_y(\theta \partial_x H)=\partial _x(\theta \partial_yH)$ and
\begin{equation}\label{previous}
\partial_y \theta \partial_x H +\theta \partial_y \partial_x H=\partial_x\theta \partial_y H+\theta \partial_x\partial_yH.
\end{equation}
Now since $H$ is harmonic and hence $\mathcal{C}^\infty(B_R)$ it holds that $\partial^2_{xy}H =\partial^2_{yx}H$. Hence $$\partial_y\theta \partial_x H = \partial_x \theta \partial_y H.$$
Thus for all $\varphi \in \mathcal{C}^\infty_c(B_R,\R)$ we have 
\begin{eqnarray}
\langle \partial_y\theta \partial_x H, \varphi \rangle & = & \langle \partial_x\theta \partial_y H, \varphi  \rangle \nonumber \\
\langle \partial_y\theta, \partial_x H\varphi \rangle & =& \langle \partial _x\theta,\partial_yH\varphi \rangle \nonumber \\
\langle \partial_x\theta,-\partial_y H\varphi \rangle + \langle \partial_y\theta, \partial_x H\varphi \rangle &=& 0 \nonumber \\
\int_ {B_R} \theta \dive \psi &=&0 \nonumber\\
\int_{B_R^+} \dive \psi - \int_{B_R^-}  \dive \psi &=&0. \nonumber
\end{eqnarray}
where in the last equalities we set $\psi:=(-\partial_yH\varphi, \partial_x H\varphi)$. We then use Theorem \ref{Gauss-Green} to obtain

\begin{equation}\nonumber
\int_{\partial_\star B_R^+} \psi \cdot \nu_{B_R^+} d\mathcal{H}^1-\int_{\partial_\star B_R^-} \psi \cdot \nu_{B_R^-} d\mathcal{H}^1=0.
\end{equation}

\noindent But $\partial_\star B_R^+=\partial_\star B_R^-$ and $\nu_{B_R^+}=-\nu_{B_R^-}$ because $B^+_R=B_R \setminus B_R^-$. Hence we obtain

\begin{equation}\nonumber
2\int_{ \partial_\star B_R^+} \psi \cdot \nu_{B_R^+}d\mathcal{H}^1=0.
\end{equation}
Using the fact that $\varphi$ has compact support in $B_R$ and the definition of $\psi$ we find that

\begin{equation}\label{ortho}
\int_{\partial_\star B_R^+ \setminus \partial B_R} (-\partial_yH \nu_{B_R^+}^1+\partial_x H\nu_{B_R^+}^2)\varphi d\mathcal{H}^1=0
\end{equation}
for all $\varphi \in \mathcal{C}_c^1(B_R,\R^2)$, where we denoted $\nu_{B_R^+}=(\nu_{B_R^+}^1, \nu_{B_R^+}^2)$.
We conclude from \eqref{ortho} that 
$$-\partial_yH\nu_{B_R^+}^1+\partial_x H\nu_{B_R^+}^2=0, \ \mathcal{H}^1-\text{a.e. on} \ \partial_\star B_R^+ \setminus \partial B_R.$$ 

\noindent The last equality means that $\nu_{B_R^+}$ is orthogonal to $(-\partial_yH,\partial_xH)$ $\mathcal{H}^1$-a.e on $\partial_\star B_R^+ \setminus \partial B_R$ and hence parallel to $(\partial_xH,\partial_yH)$. We also obtain that $\nu_{B_R^+}$ is collinear to $\nabla h$ because $\nabla h =\theta \nabla H$ in $B_R$.
\end{proof}

We can now describe the support of the measure $\mu_ {\lfloor B_R}$ in terms of the boundary of $B_R^+$.

\begin{lemma}\label{support}
Let $h$ satisfy  the hypothesis \eqref{C1}, \eqref{C2}, \eqref{C3}. Let $\theta$, $H$, $B_R$, $B_R^+$ as in  Lemma \ref{starting point}. Then the support of $\mu_{\lfloor B_R}$ is $\partial _\star B_R ^+ \setminus \partial B_R$ and we have 

$$\mu_{\lfloor B_R}= -2 \nabla H \cdot \nu_{B_R^+} \mathcal{H}^1_{\lfloor \partial_\star B_R ^+ \setminus \partial B_R}.$$ 
\end{lemma}

\begin{proof}
It holds that 

\begin{eqnarray}
\langle \Delta h,\varphi \rangle &=& -\int_{B_R} \nabla h \cdot \nabla \varphi, \ \ \forall \varphi \in \mathcal{C}^\infty_c (B_R) \ \
(h \in H^1(\Omega)) \nonumber \\
&=& -\int_{B_R^+} \nabla H \cdot \nabla \varphi + \int _{B_R^-} \nabla H \cdot \nabla \varphi \ \ (\nabla h=\pm 1 \nabla H \ \text{in} \ B_R^{\pm }) \nonumber
\end{eqnarray}
Now we use the fact that $H$ is harmonic in $B_R$ ($\Delta H=0$ in $B_R$) and the Gauss-Green formula \ref{Gauss-Green} to obtain
\begin{eqnarray}
\langle \Delta h,\varphi \rangle &=& -\int_{\partial_\star B_R^+ \setminus \partial B_R} \varphi \nabla H \cdot \nu_{B_R^+} d \mathcal{H}^1 +\int_{\partial_\star B_R^- \setminus \partial B_R} \varphi \nabla H \cdot \nu_{B_R^-} d \mathcal{H}^1  \nonumber \\
&=& -2\int_{\partial_\star B_R^+ \setminus \partial B_R} \varphi \nabla H \cdot \nu_{B_R^+} d \mathcal{H}^1 \nonumber
\end{eqnarray}
since $\partial_* B_R^+ =\partial_* B_R^-$ and $\nu_{B_R^+}=-\nu_{B_R^-}$. Now because of the previous Lemma \ref{colinear} we have that $|\nabla H \cdot \nu_{B_R^+}|=|\nabla H| \neq 0$ in $B_R$ (recall that $|\nabla H|=|\nabla h|$ in $B_R$). Hence we can deduce that the support of $\mu_{\lfloor B_R}$ is $\partial_* B_R^+ \setminus \partial B_R$ and the lemma is proved.
\end{proof}

We study in more details $\partial_\star B_R^+ \setminus \partial B_R$. In particular since  $\nabla h$ and $\nu_{B_R^+}$ are parallel on $\partial_\star B_R^+ \setminus \partial B_R$ we expect $H$ to be constant on the connected components of this set. In order to prove this fact we need more definitions and more results from geometric measure theory, these can be found in the article \cite{ACMM}. 

\begin{definition}
A curve $\Gamma \subset \R^2$ is a \textit{Jordan curve} if $\Gamma=\gamma([a,b])$ for some $a,b \in \R$ with $a<b$, and some continuous map $\gamma$, one-to-one on $[a,b)$ and such that $\gamma(a)=\gamma(b)$.
\end{definition}

\begin{definition}
A curve $\Gamma \subset \R^2$ is \textit{rectifiable} if $\mathcal{H}^1(\Gamma) < \infty$.
\end{definition}

\begin{lemma}[Lemma 3 in \cite{ACMM}]\label{Lipschitz}
Let $C \subset \R^n$ be a compact connected set with $\mathcal{H}^1(C) < \infty$. Then for any pair of distinct points $x,y \in C$ there exists a Lipschitz one-to-one map $\gamma:[0,1] \rightarrow C$ such that $\gamma(0)=x$ and $\gamma (1) =y$.
\end{lemma}
\noindent A consequence of this lemma is that any rectifiable Jordan curve admits a Lipschitz re-parametrization.

In order to state the next theorem, following \cite{ACMM},  we introduce a \textit{formal} Jordan curve $J_\infty$ whose interior is 
$\R ^n$ and a \textit{formal} Jordan curve $J_0$ whose interior is empty. We denote by $\mathcal{S}$ the set of 
Jordan curves and formal Jordan curves. We then have the following description of the essential boundary of sets 
of finite perimeter in $\R^2$.

\begin{thm}[Corollary 1 in \cite{ACMM}]\label{ACMM}
Let $E$ be a subset of $\R²$ of finite perimeter. Then there is a unique decomposition of $\partial_\star E$ into rectifiable Jordan curves $\{C_i^+, C_k^-: i,k \in \mathbb{N} \}\subset \mathcal{S}$, such that 

\begin{itemize}
\item[i)]Given $\inte(C_i^+), \inte(C_k^+), i\neq k$, they are either disjoint or one is contained in the other; given  $\inte(C_i^-), \inte(C_k^-), i\neq k$, they are either disjoint or one is contained in the other. Each $\inte(C_i^-)$ is contained in one of the $\inte(C_k^+)$.
\item[ii)]$P(E)=\sum_i \mathcal{H}^1(C_i^+) +\sum_k \mathcal{H}^1(C_k^-)$.
\item[iii)] If $\inte(C_i^+) \subset \inte(C_j^+)$, $i \neq j$, then there is some rectifiable Jordan curve $C_k^-$ such that $\inte(C_i^+) \subset \inte (C_k^-) \subset \inte(C_j^+)$. Similarly if $\inte(C_i^-) \subset \inte(C_j^-)$, $i \neq j$, then there is some rectifiable Jordan curve $C_k^+$ such that $\inte(C_i^-) \subset \inte (C_k^+) \subset \inte(C_j^-)$.
\item[iv)] Setting $L_j=\{i ; \ \inte(C_i^- \subseteq \inte(C_j^+)\}$, the sets $Y_j=\inte(C_j^+) \setminus \cup_{i\in L_j} \inte(C_i^-)$ are pairwise disjoint, indecomposable and $E=\cup_{j} Y_j$. 
\end{itemize}
\end{thm}

We are now able to prove:

\begin{lemma}\label{perimetrefini}
Let $\theta$ be such that \eqref{ImP} holds, $\theta \in BV(B_R)$. Let $B_R^+$ as before. There exist (possibly infinitely many) disjoint rectifiable Jordan curves $\gamma_i$ such that
$$\partial_\star B_R^+ =\bigcup _{i=1}^{+\infty} \gamma_i.$$
\end{lemma}

\begin{proof}

We must check that $B_R^+$ is a set of finite perimeter in $\R^2$ (not just in $B_R$) in order to apply Theorem \ref{ACMM}. To this end we set 

$$
\tilde{\theta}:= \left\{ 
\begin{aligned}
\theta  \ \text{if} \ &  x \in B_R, \\
-1   \ \text{if}  \ & x\in \R^² \setminus B_R .
\end{aligned}
\right.$$
We also set $\chi_{B_R^+}=\frac{1}{2}(1+\tilde{\theta})$, this is the characteristic function of $B_R^+$ in $\R^2$. We must prove that $\chi_{B_R^+} \in BV(\R^2)$. First we note that $\chi_{B_R^+} \in L^1(\R^2)$ because it is bounded in $B_R$ and it is null in $\R^2\setminus B_R$. Second for all $\varphi \in \mathcal{C}_c^\infty (\R^2,\R^2)$ we have

$$\int_{\R^2} \chi_{B_R^+} \dive(\varphi) = \frac{1}{2}\int_{\R^2} \tilde{\theta}\dive(\varphi)$$ 
since $\varphi$ has compact support in $\R^2$. Thus
\begin{eqnarray}
\int_{\R^2} \chi_{B_R^+} \dive(\varphi) &=& \frac{1}{2} \left[\int_{B_R} \theta \dive(\varphi) - \int_{\R^2 \setminus B_R} \dive(\varphi)\right] \nonumber \\
 &=& \frac{1}{2} \int_{B_R^+} \dive(\varphi) -\frac{1}{2}\int_{B_R^-} \dive(\varphi) -\frac{1}{2}\int_{\R^2 \setminus B_R} \dive(\varphi). \nonumber
\end{eqnarray}

Now we claim that $B_R^+$, $B_R^-$ and $\R^2 \setminus B_R$ have locally finite perimeter in $\R^2$. This is obvious for $\R^2 \setminus B_R$ because $B_R$ is a smooth open set with finite perimeter in $\R^2$. For $B_R^+$, $B_R^-$ thanks to a  deep criterion (\textit{cf.} Theorem 1 p.222 of  \cite{EvansGariepy}) we must only check that for all $K$ compact subset of $\R^2$ 
$$\mathcal{H}^1(K \cap \partial_\star B_R^+)  < +\infty.$$
But $\mathcal{H}^1(K \cap \partial_\star B_R^+)\leq \mathcal{H}^1(\overline{B_R} \cap \partial_\star B_R^+) <+\infty$ because  $B_R^+$ has finite perimeter in $B_R$ by definition, and the same is true for $B_R^-$. We can thus apply the Gauss-Green formula \ref{Gauss-Green} to obtain

\begin{eqnarray}
\int_{\R^2} \tilde{\theta} \dive(\varphi)= \int_{\partial_\star B_R^+} \varphi \cdot \nu_{B_R^+} d\mathcal{H}^1 - \int_{\partial_\star B_R^-} \varphi \cdot \nu_{B_R^-} d\mathcal{H}^1 -\int_{\partial B_R} \varphi \cdot \nu_{B_R}d\mathcal{H}^1. \nonumber
\end{eqnarray}

\noindent Hence for all $\varphi \in \mathcal{C}_c^\infty(\R^2,\R^2)$ we have 
$$|\int_{\R^2} \tilde{\theta} \dive(\varphi)|\leq 2\mathcal{H}^1(\partial_\star B_R^+)+\mathcal{H}^1(\partial B_R) < +\infty. $$

This proves that $\chi_{B_R^+}$ is in $BV(\R^2)$. We can thus apply the Theorem \ref{ACMM} to obtain the lemma.
\end{proof}



In order to pursue the proof of the main result we need the following version of the coarea formula:

\begin{thm}[Theorem 2.93 in \cite{AmbrosioFuscoPallara}  p.101]\label{coarea formula}
 Let $f:\R^2 \rightarrow \R$ be a Lipschitz function and let $E$ be a countably $\mathcal{H}^1$-rectifiable subset of $\R^2$. Then the function $t \mapsto \mathcal{H}^0(E \cap f^{-1}(t))$ is Lebesgue measurable in $\R$, $E \cap f^{-1}(t)$ is countably $\mathcal{H}^0$-rectifiable for $dt$-a.e. $t \in \R$ and 
 \begin{equation}
 \int_E C_k d^E f_x d\mathcal{H}^1(x)=\int_\R \mathcal{H}^0(E \cap f^{-1}(t))dt.
 \end{equation}

\noindent where $d^Ef_x$ is the tangential differential of $f$ at $x \in E$, $C_kd^E f_x$ is the $k$-dimensional coarea factor, $\mathcal{H}^0$ is the $0$-dimensional Hausdorff measure (this is the counting measure)  and for the definitions of these notions we refer to \cite{AmbrosioFuscoPallara} Chapter 2.
\end{thm}

We can apply the previous theorem with the function $f:\R^2 \rightarrow \R$, $x \mapsto |x|$ (we have that $|d^Ef_x| \leq 1$ for this $f$ and all $E$ countably $\mathcal{H}^1$-rectifiable subset of $\R^2$). We then find that 
for all rectifiable Jordan curves $\gamma$ we have, for $R> \rho >0$

\begin{equation}\label{coarea2}
\mathcal{H}^1(\gamma \cap (B_R \setminus B_{R-\rho})) \geq \int_{R-\rho}^R \mathcal{H}^0(\gamma_i \cap C_t)dt
\end{equation}

\noindent where $C_t=\{z \in \R^2 ; |z|=t \}$. We then obtain: 

\begin{lemma}\label{simplecurves1}
Under the same assumptions as in Lemma \ref{perimetrefini}. There exist $0<R'<R$ and (possibly infinitely many) connected rectifiable simple curves $\Gamma_j$ such that 
\begin{equation}
\partial_\star B_{R'}^+ \setminus \partial B_{R'} =\bigcup _{j=1}^{+\infty} \Gamma_j.
\end{equation}
\end{lemma}

\begin{figure}[ht!]
\begin{center}
\scalebox{0.5}{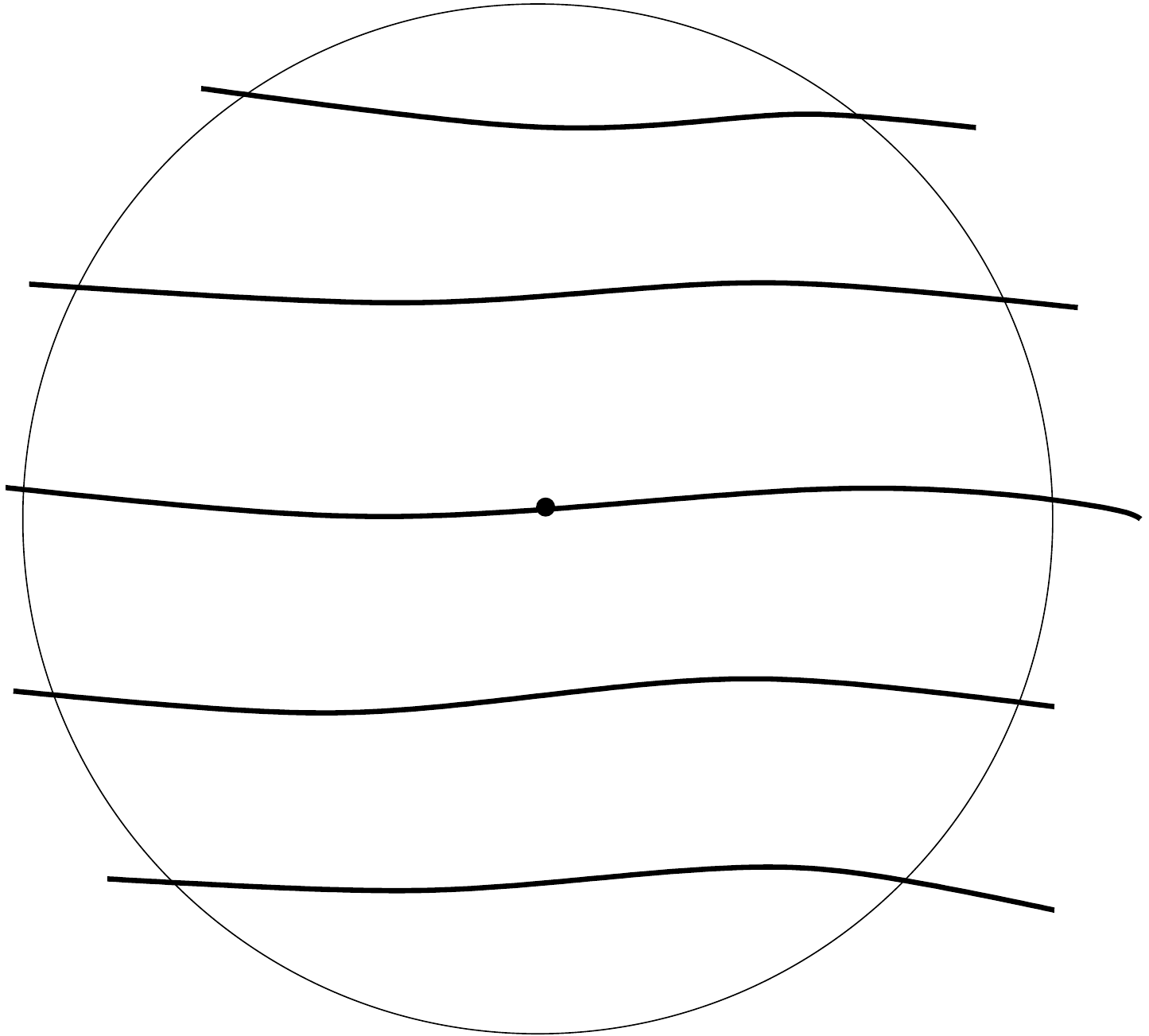}
\end{center}
\caption{Illustration of Lemma \ref{simplecurves1}}
\end{figure}

\begin{proof}
We use the formula \eqref{coarea2}, and the fact that $\mathcal{H}^0(\gamma_i \cap C_t)$ is finite for almost every $t \in [R-\rho,R]$ . We choose $R'$ such that $\mathcal{H}^0(\gamma_i \cap C_{R'}) < +\infty$ and we have
that for all $i \in \mathbb{N}$ there exists $k_i \in \mathbb{N}$ and $k_i$ intervals of $\R$ such that 
$$\gamma_i \cap B_{R'} =\gamma_i(I_1) \cup ... \cup \gamma_i(I_{k_i}) $$
with $I_j=]a_j,b_j[$ and $\gamma_i(a_j), \gamma_i(b_j) \in \partial B_{R'}$ for $j=1,...,k_i$. Hence $\gamma_i \cap  B_{R'}$ is a finite union of connected rectifiable simple curves.
We define $B_{R'}^+ =B_R^+ \cap B_{R'}$ and we find that 

\begin{eqnarray}
\supp \mu_ {\lfloor{B_{R'}}} &=& \partial_\star B_{R'}^+ \setminus \partial B_{R'}
= \supp \mu _ {\lfloor {B_R}} \cap B_{R'} \nonumber \\
&=& \bigcup _{i=1}^{+\infty} \gamma_i \cap B_{R'}  = \bigcup_{j=1}^{+\infty} \Gamma_j \nonumber
\end{eqnarray}
with $\Gamma_j$ connected rectifiable simple curves.
\end{proof}

We are now in position to prove Theorem \ref{theorem1}.

\begin{proof}[Proof of Theorem \ref{theorem1}] 
Let $h$ which satisfies \eqref{C1}, \eqref{C2},\eqref{C3}. Let $\theta$, $H$ be defined by \eqref{ImP}. Let $R'>0$ be as in Lemma \ref{simplecurves1}. From now on we denote by $B$ the ball $B_{R'}$. We also denote by $B^+=\{z\in B ; \theta(z)=+1\}$. Let $\{\Gamma_j \}_{j\in \mathbb{N}}$ simple connected rectifiable given by Lemma \ref{simplecurves1}.
The next claim states that each connected component of $\partial_\star B^+ \setminus \partial B$ is a connected component  of some level curve of the function $H$ in $B$.

\begin{claim}\label{cc}
For all $i \in \mathbb{N}$, there exists $c_i \in \R$ such that 
$$\Gamma_i =\{z \in \R^2 ; H(z)=c_i \} \cap B.$$
\end{claim}

1) We first show that for all $i \in \mathbb{N}$ there exists $c_i \in \R$ such that $\Gamma_i \subset \{H=c_i\} \cap B $, where $\{H=c_i\}$ is a short for $\{z \in \R^2 ; H(z)=c_i \}$. Indeed let $x,y \in \Gamma_i$, $x\neq y$, thanks to Lemma \ref{Lipschitz} we can find a bijective lipschitz map $f:[0,1] \rightarrow \Gamma_i$ such that $f(0)=x$ and $f(1)=y$. We then have 

$$H(y)-H(x)=\int_0^1 \frac{d}{dt}(H\circ f)(t)dt $$

\noindent because $H\circ f \in W^{1,1}([0,1],\R)$, (that is $H\circ f$ is absolutely continuous). To prove the absolute continuity we use that $H\in \mathcal{C}^\infty(\overline{B})$ and $f\in W^{1,\infty}([0,1],B)$. Hence ( we obtain that $H\circ f \in W^{1,\infty}([0,1]) \subset W^{1,1}([0,1])$  (see \textit{e.g.}  Proposition 9.5 p. 270 of \cite{Brezis}). Thus
$$H(y)-H(x)=\int_0^1 \nabla H(f(t)\cdot f'(t) dt$$
where $f'(t)$ denotes the derivative of $f$ which exists for $\mathcal{L}^1$-almost every $t \in [0,1]$ (because Lipschitz functions are differentiable almost everywhere). But $f'(t)$ is tangent to $\Gamma_i$ and  $\nabla H(f(t))$ is orthogonal to $f'(t)$ for almost every $t\in [0,1]$. Indeed thanks to Lemma \ref{colinear}, we have that $\nabla H $ parallel to $\nu_{B^+}$ $\mathcal{H}^1$-a.e. Hence we obtain that $\nabla H(f(t)) \cdot f'(t)=0$ a.e. and $H(y)=H(x)$. This shows that $\Gamma_i \subset \{H=c_i\} \cap B $.\\

2) We show that $\Gamma_i=\{H=c_i\} \cap B$ using the following Lemma \ref{connectedcurve}.
 We use the fact that since $\nabla H$ does not vanish in $B$ the level curves $\{H=c_i\} \cap B$ are diffeomorphic to straight line (this is a consequence of the implicit function theorem or can be seen in Theorem \ref{courbes harmoniques}) if $R'$ is small enough. Hence they are connected. We then apply Lemma \ref{connectedcurve} to $\Gamma_i$ and $\{H=c_i\} \cap B$. These two curves are rectifiable, connected and simple, and we have 
$\Gamma_i \subset \{H=c_i\} \cap B$ and $\Gamma_i \cap \partial B= \{H=c_i\} \cap \partial B$ by continuity of $H$.

\begin{lemma}\label{connectedcurve}
Let $B$ be a ball of radius $R$. Let $\gamma$ and $\tilde{\gamma}$ be two connected rectifiable simple curves. We also denote by  $\gamma, \tilde{\gamma} :[0,1] \rightarrow \R^2$ some  Lipschitz parametrization of these curves. We suppose that $\gamma$, $\tilde{\gamma}$ are homeomorphism from $[0,1]$ onto their image. Assume that  
\begin{itemize}
\item[i)] $\gamma (]0,1[) \subset B$ and $\gamma(0),\gamma(1) \in \partial B,$
\item[ii)] $\tilde{\gamma} (]0,1[) \subset B$ and $\tilde{\gamma}(0),\tilde{\gamma}(1) \in \partial B,$
\item[iii)] $\tilde{\gamma}([0,1]) \subset \gamma ([0,1]).$ 
\end{itemize}

Then $\gamma = \tilde{\gamma}$.
\end{lemma}

We postpone the proof of this lemma at the end of the section. Now that we know the geometry of the curves $\Gamma_i$ we can prove that there exists only a finite number of such curves in a sufficiently small ball.   

\begin{claim}\label{finite}
Let $\rho>0$ small enough such that $\Gamma_i \cap B(z_0,\rho) =\{H=c_i\} \cap B(z_0,\rho)$ is diffeomorphic to an open segment for all $i \in \mathbb{N}$ such that $\Gamma_i \neq \emptyset$. Then there exists a finite number of curves $\Gamma_i$ such that $ \Gamma_i \cap B(z_0,\rho) \neq \emptyset$.
\end{claim}

With $\rho$ as in the statement of the claim we let $B_\rho=B(z_0,\rho)$. Since $\theta \in BV(B_R)$ we also have $\theta \in BV(B_\rho)$.Thus using the same notations as before we have 

\begin{eqnarray}
 +\infty & > & \mathcal{H}^1(\partial_\star B_\rho^+ \setminus \partial B_\rho) \nonumber \\
  & = & \mathcal{H}^1(\supp(\mu_{\lfloor B_\rho})) \nonumber \\
  & =& \mathcal{H}^1 ( \bigcup_{i=1}^{+\infty} \Gamma_i \cap B_\rho) \nonumber \\
  &\geq & \int_0^\rho \mathcal{H}^0 (\bigcup_{i=1}^{+\infty} \Gamma_i \cap C_t) dt \nonumber
\end{eqnarray}
where in the last equality we used the coarea formula (Theorem \ref{coarea formula}), and we let $C_t=\{z\in \C ; |z|=t\}$.
The coarea formula also tells us that for almost every $t \in [0,\rho]$ we have 
$\mathcal{H}^0 (\bigcup_{i=1}^{+\infty} \Gamma_i \cap C_t) < +\infty $. But if $\rho$ is small 
enough then every level curves of the harmonic function $H$ meet the boundary of the ball $B_\rho$. This is a 
consequence of the maximum principle.  
As a consequence we have that $\mathcal{H}^0 (\bigcup_{i=1}^{+\infty} \Gamma_i \cap C_t)$ is 
exactly two times the number of curves $\Gamma_i$ inside $B_t$. Then the number of curves $\Gamma_i$ is finite 
inside $B_\rho$. \\

We can now conclude the proof of Theorem \ref{theorem1}. The last claim proved that there exists a 
finite number of $\Gamma_i$ near $Z_0:=\{z \in B ; H(z)=0 \}$. Thus there exists $\eta >0$ such that $\dist (Z_0,
\Gamma_i)> \eta$ for all $i\in \mathbb{N}$ such that $\Gamma_i$ is not included in $Z_0$. We then set $V:= B(z_0,
\frac{\eta}{2})$. Because of the definition of $\eta$ we obtain that
$$\supp (\mu_{\lfloor V})=\{z \in ; H(z)=0 \}.$$
Note that $Z_0$ is a smooth connected rectifiable curve near $z_0$ (since $\nabla H(z_0) \neq 0$). 
We also set as usual $V^+ =\{z\in V ;\theta(z)=+1\}$, $V^{-}=\{z\in V ;\theta(z)=-1\}$. We have that 
$$\nabla h =+\nabla H, \ \ \text{on} \ V^+, \ \ \ \nabla h =-\nabla H, \ \ \text{on} \ V^- .$$
We thus deduce that 
$h=H$ on $V^+$ and $h=-H$ on $V^-$ because $h=H=0$ on $\partial_\star V^+ \setminus \partial V= Z_0$. We know that $H$ does not vanish in $V^+$ and $V^-$, because $H$ vanishes only on $Z_0$. Hence $H$ has constant sign on $V^+$ and on $V^-$ thanks to the maximum principle. These two signs are opposite, because if they were the same then the minimum (or maximum) of $H$ would be $0$ and would be inside the domain $V$, this contradicts the maximum principle. We can assume for example that $H$ is non negative in $V^+$ and then  $h=|H|$ in $V$.  

\end{proof}

\begin{proof}[Proof of Lemma \ref{connectedcurve}]
By contradiction, assume that there exists $ p \in \gamma \setminus \tilde{\gamma}$. Let $ t_0 \in ]0,1[$ such that $\gamma(t_0)=p$. Then we have 
$$]0,1[=\tilde{\gamma}^{-1} \big( \gamma (]0,t_0[) \cup \gamma (]t_0,1[) \big) $$
since $\tilde{\gamma}(]0,1[) \subset \gamma (]0,1[)$ and since $\gamma(t_0) \notin \tilde{\gamma}$. We then deduce that
$$]0,1[=\tilde{\gamma}^{-1}\left(\gamma (]0,t_0[)\right) \cup \tilde{\gamma}^{-1}\left(\gamma (]t_0,1[)\right).$$
But since $\tilde{\gamma}$ and $\gamma$ are homeomorphism onto their image we have that $\tilde{\gamma}^{-1}(\gamma (]0,t_0[))$ and $\tilde{\gamma}^{-1}(\gamma (]t_0,1[))$ are two disjoint open sets. Thanks to the connectedness of $]0,1[$ we can deduce that 

\begin{itemize}
\item[1)] $\tilde{\gamma}^{-1}(\gamma(]0,t_0[))=]0,1[$ and $\tilde{\gamma}^{-1}(\gamma(]t_0,1[))= \emptyset$ or
\item[2)] $\tilde{\gamma}^{-1}(\gamma(]0,t_0[))= \emptyset$ and $\tilde{\gamma}^{-1}(\gamma(]t_0,1[))= ]0,1[$.
\end{itemize}
These two cases are similar. Let us assume that we are in case 1). We can then obtain that 
$$\gamma (]t_0,1[) \cap \tilde{\gamma}(]0,1[)= \emptyset .$$
This implies that $\gamma (]t_0,1[)=\emptyset$ or $\tilde{\gamma} \nsubseteq \gamma$. The first assertion is impossible because $\gamma$ is assumed to be a homeomorphism from $[0,1]$ onto its image and the second possibility is in contradiction with the hypothesis iii). Thus it holds that $\tilde{\gamma}=\gamma$.
\end{proof}

\section{Second case: local behavior near a zero of even order of $\omega_{\h}(z)=(\partial_x\h-i\partial_y\h)^2(z)$}

This section is devoted to the proof of Theorem \ref{theorem2}. It is very similar to the proof of Theorem \ref{theorem1}. Here $\omega_{\h} (z_0)=0$, but since we assume that $z_0$ is a zero of even order of $\omega_{\h}$ there is no difficulty to find a holomorphic function $g$ such that $(\partial_x\h-i\partial_y\h)^2= g(z)^2.$
Then the proof of Theorem \ref{theorem2} is a rather direct adaptation of the proof of Theorem \ref{theorem1} 
except that here because the function $g$ vanishes at $z_0$ we can only show that the function $\theta$ defined 
as in the previous section is in $BV_{loc}(B_R \setminus\{z_0\})$ for $R$ sufficiently enough. This introduce a 
new technical difficulty. We drop the subscript $\mu$ in the rest of this section. \\

\begin{lemma}\label{starting point2}
Let $h$ which satisfies \eqref{C1}, \eqref{C2}, \eqref{C3}. Let $z_0\in \Omega$ be a zero of even order of $\omega_h(z)=(\partial_xh-i\partial_yh)^2(z)$ . Then there exist $R>0$, a function $\theta: B_R \rightarrow \{ \pm 1\}$ and a harmonic function $H:B_R \rightarrow \R$ such that 
\begin{equation}\label{ImP2}
\partial_xh(z) -i \partial_y h(z)= \theta(z)\left(\partial_xH(z)-i\partial_yH(z)\right), \ \ \forall \ z \in B_R
\end{equation}
\end{lemma}

\begin{proof}

\noindent Since $z_0$ is a zero of even order of $\omega_h$, we can find a neighborhood $U$ of $z_0$, $n\in \mathbb{N}$ and a holomorphic function $f_1: U \rightarrow \C$ such that $f_1(z_0)\neq 0$ and

\begin{equation}\label{equality1}
(\partial_xh-i\partial_yh)^2=(z-z_0)^{2n}f_1(z).
\end{equation}

\noindent Since $f_1(z_0)\neq 0$, we can choose a smaller neighborhood of $z_0$ still denoted by $U$ such that in $U$ there exists a holomorphic function denoted by $\varphi_1$ which satisfies $\varphi_1^2(z)=f_1(z)$, and furthermore we can choose $U=B(z_0,R)$ for $R$ small enough. We then have 
\begin{equation}\label{defg}
(\partial_xh-i\partial_yh)^2=[(z-z_0)^n\varphi(z)]^2=:g(z)^2.
\end{equation}

\noindent We set $F(z):= \int_{z_0}^z g(s)ds $ and 
\begin{equation}\label{defH}
H(z):=\text{Re} F(z)= \text{Re} \left(\int_{z_0}^z g(s)ds\right).
\end{equation}
\noindent The function $H$ is harmonic in $B_R$ and satisfies $$2\partial_z H =F'(z)=g(z)=(z-z_0)^n\varphi_1(z). $$

\noindent Hence, thanks to \eqref{equality1} we deduce that there exists $\theta: U \rightarrow \{ \pm1\}$ such that 

\begin{equation}\label{deftheta}
\partial_xh-i\partial_yh = \theta(z)(\partial_xH-i\partial_yH)
\end{equation}

\end{proof}

As before we set 

$$B_R^+ := \{z \in B_R ; \theta(z)=+1 \},\ \ 
B_R^-:=\{z \in B_R ; \theta(z)=-1 \} .$$

\noindent We thus obtain that 

$$\nabla h = +\nabla H, \ \text{on} \ B_R^+, \ \ \nabla h= -\nabla H, \ \text{on} \ B_R^{-}.$$

\begin{lemma}\label{lemma2.1}
Let $h$ which satisfies \eqref{C1}, \eqref{C2}, \eqref{C3}. Let $R >0$ be small enough and $\theta:B_R \rightarrow \{\pm 1\}$ such that \eqref{deftheta} holds with $H$ defined by \eqref{defH}. Then the function $\theta$ is in $BV_{loc}(B_R \setminus \{z_0\}) $.
\end{lemma}

\begin{proof}
In order to prove this result we can apply Lemma \ref{lemma1} of the previous section in any open subset $W\subset B_R$ such that $g=\partial_xH-i\partial_yH$ does not vanish in $W$.  
\end{proof}

\noindent Thus $B_R^+$ and $B_R^-$ are sets of locally finite perimeter in $B_R\setminus\{z_0\}$.

\begin{lemma}\label{collinear2}
Let $h$ as in Lemma \ref{starting point2}. Let $\theta$, $H$ given by Lemma \ref{starting point2}. Let $B_R^+=\{z \in B_R; \theta(z)=+1\}$, thanks to Lemma \ref{lemma2.1} $B_R^+$ is a set of locally finite perimeter in $B_R\setminus (\{z_0\})$. Furthermore the generalized normal $\nu_{B_R^+}$ is collinear to $\nabla h$ and $\nabla H$, $\mathcal{H}^{1}_ {\lfloor{\partial_\star B^+_R \setminus \partial B_R}}$ almost everywhere in $B_R$.
\end{lemma}

The proof of this lemma is exactly the same as the one of Lemma \ref{colinear}. We can also copy the proof of Lemma \ref{support} to obtain

\begin{lemma}\label{support2}
Let $h$ satisfy the hypothesis \eqref{C1}, \eqref{C2}, \eqref{C3}. Let $\theta$, $B_R$, $B_R^+$ as in the previous lemma \ref{colinear}. Then the support of $\mu_{\lfloor B_R}$ is $\partial _\star B_R ^+ \setminus \partial B_R$ and we have 

$$\mu_{\lfloor B_R}= -2 \nabla H \cdot \nu_{B_R^+} \mathcal{H}^1_{\lfloor \partial_* B_R ^+ \setminus \partial B_R}.$$ 
\end{lemma}

We would like to apply Theorem \ref{ACMM} to the set $B_R^+$ and continue the proof as in the previous section but we can not do that because $B_R^+$ have only \textit{locally} finite perimeter in $B_R \setminus \{z_0\}$. In fact we will show that this is just a technical issue and that $B_R^+$ has indeed finite perimeter in $B_R$ but it requires some work. In a first time we work in an annular domain. Let $0<\rho<R$, we set
\begin{eqnarray}
A_{R,\rho} =\{z \in \C ; \rho < |z| < R \} \ \ \
A_{R,\rho}^+ =\{ z \in A_{R,\rho} ; \theta(z)=+1 \}. \nonumber
\end{eqnarray}

We first apply Theorem \ref{ACMM} to the set $A_{R,\rho}^+$.

\begin{lemma}\label{jordancurvesannulus}
Let $\theta$ be such that \eqref{ImP2} holds, $\theta \in BV_{loc}(B_R\setminus\{z_0\})$. Let $A_{R,\rho}^+$ as before. There exist (possibly infinitely many) disjoint rectifiable Jordan curves $\gamma_i^\rho$ such that
$$\partial_\star A_{R,\rho}^+ =\bigcup _{i=1}^{+\infty} \gamma_i^\rho.$$
\end{lemma}

\begin{proof}
We have that $\theta \in BV (A_{R,\rho})$. As in the proof of Lemma \ref{perimetrefini} one can show that $A_{R,\rho}^+$ has finite perimeter in $\R^2$. We can then apply Theorem \ref{ACMM} to deduce the result.
\end{proof}

\begin{lemma}\label{simplecurvesannulus}
Under the same assumptions as Lemma \ref{jordancurvesannulus}, there exist $0<\rho<\rho'< R'<R$ and (possibly infinitely many) connected rectifiable simple curves $\Gamma_j^\rho$ such that 
\begin{equation}
\partial_\star(A_{R',\rho'}^+ ) \setminus \partial (A_{R',\rho'}) =\bigcup _{j=1}^\infty \Gamma_j^{\rho'}
\end{equation}
\end{lemma}

\begin{figure}[ht!]
\begin{center}
\scalebox{0.5}{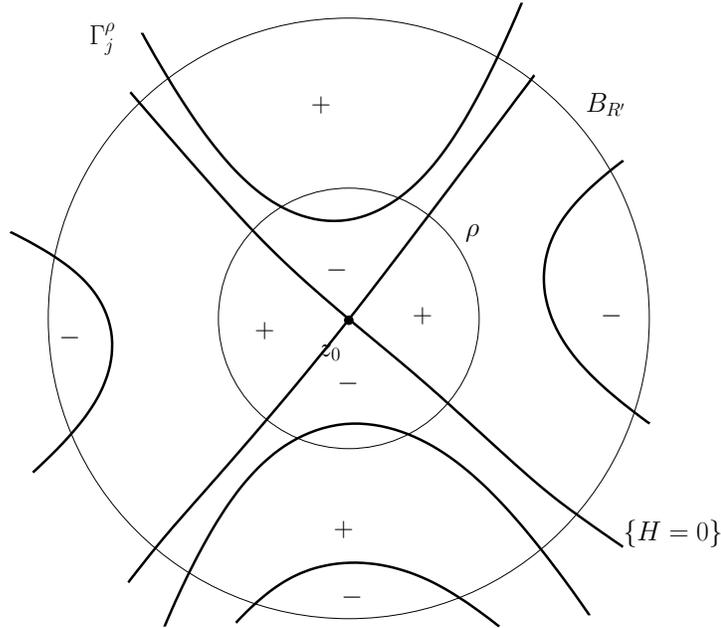}
\end{center}
\caption{Illustration of Lemma \ref{simplecurvesannulus}}
\end{figure}

\begin{proof}
The proof is the same as in Lemma \ref{simplecurves1}, it uses the coarea formula (see \ref{coarea formula}).
\end{proof}

\begin{lemma}
Under the same assumptions as in Lemma \ref{jordancurvesannulus}, for all $j \in \mathbb{N}$, there exists $c_j^{\rho'} \in \R$ such that 
$\Gamma_i^{\rho'}$ is exactly one connected component of  $\{z \in \R^2 ; H(z)=c_i^{\rho'} \} \cap A_{R',\rho'}$.
\end{lemma}

\begin{proof}
Again in order to prove this lemma we can follow line by line the proof of claim \ref{cc}. The only difference is that here $ \{z \in \R^2 ; H(z)=c_i \} \cap A_{R',\rho'}$ is not necessarily connected if $c_i \neq 0$.
\end{proof}

\begin{lemma}
Under the same assumptions as in Lemma \ref{simplecurvesannulus} with $R'$ sufficiently small there exists a finite number $N_{\rho'}$ of curves $\Gamma_j^{\rho'}$ such that $\Gamma_j^{\rho'} \cap A_{R',\rho'} \neq \emptyset$. We then have 
\begin{eqnarray} 
\supp(\mu_{\lfloor_{A_{R',\rho'}}}) &=& \partial_\star (A_{R',\rho'}^+) \setminus \partial A_{R',\rho'} \nonumber \\
& =& \bigcup _{j=1}^{N_{\rho'}} \Gamma_j^\rho . \nonumber
\end{eqnarray}
\end{lemma}

\begin{proof}
As in Lemma \ref{finite} this is due to the coarea formula and the fact that the curves $\Gamma_i^{\rho'}$ are level curves of the harmonic function $H$.
\end{proof}

The next result shows that, with $R'$ fixed if we take a larger annulus, then the number of curves in the decomposition of the support of $\mu$ is the same. This is due to the geometry of these curves since they are level curves of the harmonic function $H$.
\begin{lemma}\label{indN}
Under the same assumptions as in Lemma \ref{simplecurvesannulus} let $\rho'_1 < \rho'_2$ and $R'$ as before small enough. Then using the previous notations we have 
$N_{\rho'_2} = N_{\rho'_1}$ and, up to re-order it holds $\Gamma_j^{\rho'_1} \subset \Gamma_j^{\rho'_2}$ for $j=1,..., N_{\rho'_1}$.
\end{lemma}

\begin{proof}
Using the previous notations we have:
$$\supp (\mu_{\lfloor A_{R,\rho_1}})=\bigcup _{j=1}^{N_{\rho_1}} \Gamma_j^{\rho_1} $$
$$\supp (\mu_{\lfloor A_{R,\rho_2}})=\bigcup _{j=1}^{N_{\rho_2}} \Gamma_j^{\rho_2}. $$
Besides it holds that $\supp(\mu_{\lfloor A_{R,\rho_2}})=\supp(\mu_{\lfloor A_{R,\rho_1}}) \cap \overline{A_{R,\rho_2}}$. We thus deduce that 
$$\bigcup _{j=1}^{N_{\rho_1}} \Gamma_j^{\rho_1} \subset \bigcup _{j=1}^{N_{\rho_2}} \Gamma_j^{\rho_2}.$$

\noindent We also recall that we have the existence of real numbers $(c_j^{\rho_1}), j=1,..., N_{\rho_1}$ and $(c_j^{\rho_2}), j=1,..., N_{\rho_2}$ such that $\Gamma_j^{\rho_1}$ is exactly one connected component of $\{H=c_j^{\rho_1} \} \cap A_{R,\rho_1}$ and  $\Gamma_j^{\rho_2}$ is exactly one connected component of $\{H=c_j^{\rho_2} \} \cap A_{R,\rho_2}$. \\

\noindent Assume that there exists $c_{j_0}^{\rho_1}$ which is different from all the $c_i^{\rho_2}$ for $i=1,...,N_{\rho_2}$. Thanks to the maximum principle every connected component of level curves of the harmonic function $H$ which lies in the ball $B_R$ meets the boundary of this ball if $R$ is small enough. We thus obtain that $\Gamma_{j_0}^{\rho_1} \cap A_{R,\rho_2} \neq \emptyset$ and then 
$$\Gamma_{j_0}^{\rho_1} \cap A_{R,\rho_2} \subset \supp(\mu_{\lfloor A_{R,\rho_2}}).$$

\noindent As a consequence we obtain that $\Gamma_{j_0}^{\rho_1} \cap A_{R,\rho_2}= \Gamma_{i_0}^{\rho_2}$ for some $1\leq i_0 \leq N_{\rho_2}$. This is a contradiction with our hypothesis on $c_{j_0}^{\rho_1}$. We then have 
$$ \{c_j^{\rho_1} \}_j =\{c_i^{\rho_2} \}_i .$$
With the same justification we prove that $N_{\rho_1}=N_{\rho_2}$. And then up to reorder we have 
$$ \Gamma_j^{\rho'_1} \subset \Gamma_j^{\rho'_2} \ \text{for} \ j=1,..., N_{\rho'_1}.$$
\end{proof}

We are now in position to prove:

\begin{lemma}
Let $\theta$, $H$ be such that \eqref{ImP2} holds with $h$ which satisfies \eqref{C1}, \eqref{C2}, \eqref{C3}. As before we set $B_R^+=\{z\in B_R; \theta(z)=+1\}$. Then there exists $R>0$ small enough such that $\mathcal{H}^1(\partial_\star B_R^+ \setminus \partial B_R) <+ \infty$ and consequently
$\theta$ is in $BV(B_R)$.
\end{lemma}

\begin{proof}
By contradiction if $\mathcal{H}^1(\partial_\star B_{R}^+ \setminus \partial B_{R} )= +\infty$ then for all sequence $(\rho_n)$ of real numbers such that $\rho_n \searrow 0$ we have 
\begin{equation}\label{limit2}
\lim_{n\rightarrow +\infty} \mathcal{H}^1(\partial_\star A_{R,\rho_n}^+ \setminus \partial A_{R,\rho_n}) =+\infty.
\end{equation}
\noindent with $A_{R,\rho}=\{z \in \C ; \rho< |z| < R\}$ and $A_{R,\rho^+}= A_{R,\rho} \cap B_R^+$.
This is due to the fact that 

\begin{eqnarray}
\partial_\star B_{R}^+ \setminus \partial B_{R}& =&\bigcup _{n \in \mathbb{N}} \partial_\star B_{R}^+ \cap A_{R,\rho_n } \nonumber \\
&=& \partial_\star A_{R,\rho_n}^+ \setminus \partial A_{R,\rho_n} \nonumber
\end{eqnarray}
and the union of these sets is increasing. We now use the previous Lemmas \ref{simplecurvesannulus} and \ref{indN} to obtain that for $R$ small enough there exists an integer $N$ such that for all $n \in \mathbb{N}$ there are $N$ simple connected rectifiable curves  $\Gamma_j^{\rho_n}$ and $N$ real numbers $c_j$ such that 
$$\partial_\star A_{R,\rho_n}^+ \setminus \partial A_{R,\rho_n}=\bigcup_{j=1}^N \Gamma_j^{\rho_n}$$
and $\Gamma_j^{\rho_n}=\{H=c_j\} \cap A_{R,\rho_n}$. Furthermore we also have $\Gamma_j^{\rho_n} \subset \Gamma_j^{\rho_m}$ if $n>m$. We then obtain

\begin{eqnarray}
\H^1\left(\partial_\star B_{R}^+ \cap A_{R,\rho_n}\right) & \leq & \H^1\left(\bigcup_{i=1}^N \{H=c_i\} \cap B_{R}\right) \nonumber \\
& \leq & \sum_{i=1}^N \H^1\left(\{H=c_i\} \cap B_{R}\right). \nonumber
\end{eqnarray} 
But for $R$ small enough the level curves of $H$ have a finite Hausdorff measure. Thus there exists $M>0$ such that for all $n\in \mathbb{N}$,
$$\H^1(\partial_\star B_{R}^+ \cap A_{R,\rho_n} ) \leq M.$$
This is a contradiction with \eqref{limit2} and then $\mathcal{H}^1(\partial_\star B_R^+ \setminus \partial B_R) <+ \infty$. \\

Now we prove that $\theta \in BV(B_{R})$. We recall that in the proof of Lemma \ref{lemma2.1} we found that
$$\partial_x\theta -i\partial_y\theta =\frac{\Delta h}{g}=\frac{1}{g}\mu $$
in the sense of distributions where $g=\partial_xH-i\partial_yH$. We then have
$$\partial_x\theta -i\partial_y \theta =\frac{1}{|g|^2}\overline{g}\mu =\frac{(\partial_xH+i\partial_yH)\mu}{|\nabla H|^2}$$
thus $\partial_x\theta =\frac{\partial_xH}{|\nabla H|^2}\mu$ and $\partial_y\theta =\frac{\partial_yH}{|\nabla H|^2}\mu$. Now for all $\varphi \in \mathcal{C}^1_c(B_{R'},\R^2)$ with $|\varphi| \leq 1$
\begin{eqnarray}
\int_{B_{R}}\theta \dive  \varphi &=&-\langle \partial_x\theta,\varphi_1 \rangle -\langle \partial_y\theta,\varphi_2\rangle \nonumber \\
&=& -\int_{B_{R}} \frac{\partial_xH}{|\nabla H|^2}\varphi_1 d\mu +\int_{B_{R}} \frac{\partial_yH}{|\nabla H|^2}\varphi_2 d\mu. \nonumber
\end{eqnarray}
We now use Lemma \ref{support2} to say that 
$$\mu_{\lfloor B_{R}} =-2\nabla H\cdot\nu_{B_{R}}^+ \H^1_{\lfloor \partial _\star B_{R}^+ \setminus \partial B_{R}}$$
hence 
\begin{eqnarray}
\int_{B_{{R}}}\theta \dive  \varphi &=& 2 \int_{\partial _\star B_{R}^+ \setminus \partial B_{R}} \nabla H\cdot \nu_{B_{R}^+} \frac{\partial_xH \varphi_1}{|\nabla H|^2}d\H^1 
-2 \int_{\partial _\star B_{R}^+ \setminus \partial B_{R}} \nabla H\cdot \nu_{B_{R}^+} \frac{\partial_yH \varphi_2}{|\nabla H|^2}d\H^1 .\nonumber
\end{eqnarray}
We thus deduce, using the fact that $|\frac{\partial_xH}{|\nabla H|^2} |\leq 1$ and $|\frac{\partial_yH}{|\nabla H|^2} |\leq 1$, that 
$$|\int_{B_{{R}}}\theta \dive  \varphi | \leq 4\H^1(\partial _\star B_{R}^+ \setminus \partial B_{R}) <+\infty $$
for all $\varphi \in \mathcal{C}^1_c(B_{R},\R^2)$ ; $|\varphi|\leq 1$. This proves the claim. 
\end{proof}

From this point we have all the ingredients to pursue the proof of Theorem \ref{theorem2}  as in the previous section. 
\begin{proof}[Proof of Theorem \ref{theorem2}]

\begin{claim}
There exist $R>0$ small enough, a finite number $N$ and $N$ simple, connected, rectifiable curves $\Gamma_j$ such that 
\begin{eqnarray}
\supp(\mu_{\lfloor _{B_{R}}})=  \partial_\star B_{R}^+ \setminus \partial B_{R} 
= \bigcup_{j=1}^N \Gamma_j .\nonumber 
\end{eqnarray}

Furthermore for all $1 \leq j \leq N$ there exists $c_j$ such that $\Gamma_j$ is exactly a connected component of the level set $\{z \in \C , H(z)=c_j \} \cap B_{R}$.\\
\end{claim}

Once we know that the function $\theta$ defined by \eqref{deftheta} is in $BV(B_R)$ for $R$ small enough we can apply the same arguments as in the previous section to prove this claim, it results from an adaptation of Lemmas \ref{support}, \ref{perimetrefini}, \ref{simplecurves1}, and Claims \ref{cc}, \ref{finite}.\\

We can now conclude the proof of Theorem \ref{theorem2}. The last claim proves that there exist a finite number of $\Gamma_j$ near $Z_0:=\{z \in B ; H(z)=0 \}$. Thus there exists $\eta >0$ such that $\dist (Z_0,\Gamma_i)> \eta$ for all $j\in \mathbb{N}$ such that $\Gamma_j$ is not included in $Z_0$. We then set $V:= B(z_0,\frac{\eta}{2})$. Because of the definition of $\eta$ we can say that 
$$\supp (\mu_{\lfloor V})\subset \{z \in ; H(z)=0 \}. $$

\noindent We also set as usual $V^+ =\{z\in V ;\theta(z)=+1\}$ ,$V^{-}=\{z\in V ;\theta(z)=-1\}$. We have that 
$$\nabla h =+\nabla H, \ \ \text{on} \ V^+, \ \ \nabla h =-\nabla H, \ \ \text{on} \ V^-.$$

\noindent Note that in $V$ the function $\theta H$ is continuous since $H=0$ at the discontinuity points of $\theta$. Then $\theta H$ is in $H^1(V)$ since $H$ is in $H^1$. Computing $\nabla (\theta H)$ in the sense of distributions we  obtain that $\nabla (\theta H)=\theta \nabla H$. Besides it comes 
$$\nabla (h-\theta H )=0 \ \text{in} \ V.$$
This proves that $h-\theta H$ is constant in $V$, but evaluating this constant in $z_0$ we find that 
$$h=\theta H, \ \text{in} \ V.$$
This concludes the proof of Theorem \ref{theorem2}.

\end{proof}
\section{Third case: local behavior near a zero of odd order of $\omega_{\h}(z)=(\partial_x\h-i\partial_y\h)^2(z)$}

In this section we deal with the case where $z_0$ is a point in the support of $\mu$ and 
$z_0$ is a zero of odd order of  $\omega_{\h}$. This case is the most difficult. Indeed unlike the previous 
cases we can not find a holomorphic function $g$ such that $(\partial _x \h -i 
\partial_y \h)^2 =g^2$. We must use \textit{multivalued} holomorphic function to 
overcome this difficulty. We do not want to discuss here the notion of multivalued 
function. For us the prototype of multivalued function is $z \mapsto z^{\frac{1}{2}}$. 
Such a multivalued function is single-valued up to a sign. Indeed given any complex 
number $z$ different from $0$ there exist exactly two complex numbers $z_1$ and $z_2$ such 
that $z_i^2=z$ for $i=1,2$ and $z_1=-z_2$. In particular $|z^{\frac{1}{2}}|=|z|
^{\frac{1}{2}}$ is well defined. We drop the subscript $\mu$ during the rest of this section. \\

\begin{lemma}\label{starting point3}
Let $h$ which satisfies \eqref{C1}, \eqref{C2}, \eqref{C3}. Let $z_0\in \Omega$ be a zero of odd order of $\omega_h(z)=(\partial_xh-i\partial_yh)^2(z)$. Then there exist $W$ a neighborhood of $z_0$, a function $\theta: W \rightarrow \{ \pm 1\}$ and a function $H:W\rightarrow \R$ which satisfies $H=|H_1|$, where $H_1$ is a multivalued function $W$ such that 
\begin{equation}\label{ImP3}
\partial_xh(z) -i \partial_y h(z)= \theta(z)(\partial_xH(z)-i\partial_yH(z)) \ \text{in} \ W.
\end{equation}
Furthermore the function $H_1$ is such that: there exist an unique integer $n \geq 1$, a small number $r >0$ and a biholomorphism $\Phi :B(0,r) \rightarrow W$ such that $\Phi(0)=z_0$ and 
\begin{equation}
H_1 \circ \Phi (z)= \text{Re} (z^{n+\frac{1}{2}}), \ \ \text{for} \ z \in B(0,r)
\end{equation}
\end{lemma}

\begin{proof}

Let $z_0$ be a zero of odd order of $(\partial_x h-i\partial_y h)^2$. We can find a neighborhood $U$ of $z_0$, an integer $n$ and a holomorphic function $f_1:U \rightarrow \C$ with $f_1(z_0) \neq 0$ such that 
$$ (\partial_x h-i\partial_y h)^2=(z-z_0)^{2n+1}f_1(z) .$$

\noindent Since $f_1(z_0) \neq 0$, there exists a smaller neighborhood of $z_0$, still denoted by $U$ and a holomorphic function $\varphi_1:U \rightarrow \C$ such that $ \varphi_1 ^2(z)= f_1(z), \ \ \text{for} \ z\in U.$

\noindent We then set 
\begin{equation}\label{definitiong3}
g(z) =(z-z_0)^{n+\frac{1}{2}} \varphi_1(z), \ \ \text{for} \ z \in U.
\end{equation}
  
Like $z \mapsto z^{\frac{1}{2}}$, $g$ is a multi-valued function which is single-valued 
up to a sign. As in the previous sections we want to take a primitive of the function $g$. However the 
fact that $g$ is multivalued introduces a difficulty in this process. But we can show 
that we can choose a special form of a primitive of $g$.

\begin{claim}\label{defphi2}
There exist a neighborhood $U$ of $z_0$ and a single-valued holomorphic function $\varphi_2: U \rightarrow \C$ such that $\varphi_2(z)\neq 0$ for all $z\in  U$  and
$$G(z):= (z-z_0)^{n+\frac{3}{2}}\varphi_2(z), $$
satisfies $G'(z)=g(z)$, for all $z\in U$ (where $g$ is defined by \eqref{definitiong3}) .
\end{claim}

\begin{proof}
Let us assume that such a function $\varphi_2$ exists. We then have:
\begin{eqnarray}
G'(z)&=&[(z-z_0)^{n+\frac{3}{2}}\varphi_2(z)]' \nonumber \\
&=& (n+\frac{3}{2})(z-z_0)^{n+\frac{1}{2}}\varphi_2(z) + (z-z_0)^{n+\frac{3}{2}}\varphi_2'(z) \nonumber
\end{eqnarray}
Since we want $G'(z)=g(z)=(z-z_0)^{n+\frac{1}{2}}\varphi_1(z)$, the function $\varphi_2$ must satisfies the following complex ordinary differential equation:

\begin{equation}\label{equadiff}
(n+\frac{3}{2})(z-z_0)^{n+\frac{1}{2}}\varphi_2(z) + (z-z_0)^{n+\frac{3}{2}}\varphi_2'(z)=(z-z_0)^{n+\frac{1}{2}}\varphi_1(z).
\end{equation}
In a neighborhood of $z_0$ we can expand $\varphi_1$ in power series 
$$\varphi_1(z)=\sum_{k=0}^{+\infty}a_k(z-z_0)^k.$$
Thanks to an expansion in power series we have: $\varphi_2(z)=\sum_{k=0}^{+\infty} b_k(z-z_0)^k$. Using \eqref{equadiff} we find that the coefficient $b_k$ must satisfy

\begin{equation}\label{equadiffcoeff}
(n+\frac{3}{2})\sum_{k=0}^{+\infty} b_k(z-z_0)^k + \sum_{k=1}^{+\infty}kb_k(z-z_0)^k = \sum_{k=0}^{+\infty} a_k (z-z_0)^k.
\end{equation}

\noindent Thus we must have 
\begin{itemize}
\item[*] for $k=0$: $b_0=\frac{a_0}{n+\frac{3}{2}}$
\item[*] for $k\geq 1$: $b_k=\frac{a_k}{n+\frac{3}{2}+k}$.
\end{itemize}

\noindent We can check that if we set $\varphi_2(z)= \displaystyle{\sum_{k=0}^{+\infty} \frac{a_k}{n+3/2+k}(z-z_0)^k}$ then $G(z)=(z-z_0)^{n+\frac{3}{2}} \varphi_2(z)$ is a primitive of $g$. furthermore because $\varphi_(z_0)\neq 0$ we have $a_0\neq 0$ and hence $\varphi_2(z_0)\neq 0$. Thus $\varphi_2(z) \neq 0$ in a neighborhood of $z_0$ denoted by $U$.
\end{proof}

We then set $$H_1(z)= \text{Re}(G(z))$$ and 
\begin{equation}\label{defH3}
H(z)=|H_1(z)| \ \text{for all} \ z \ \text{in} \ U.
\end{equation}
 Note that $H_1$ is a multi-valued function which is single-valued up to a sign. The proof of the next claim is very similar to an analogous result for harmonic function (see \textit{e.g.} \cite{fonctionharmonique}).

\begin{claim}\label{levelcurvesmulti}
There exist a neighborhood $W$ of $z_0$, a number $r>0$ and an analytic diffeomorphism $\Phi:B(0,r) \rightarrow W$ such that $\Phi(0)=z_0$ and 
$$H_1 \circ \Phi(z) =\text{Re} (z^{n+\frac{3}{2}}),$$
for all $z\in B(0,r)$.
\end{claim}

\begin{proof}
We have $H_1(z)=\text{Re}(G(z))=\text{Re} [(z-z_0)^{n+\frac{3}{2}}\varphi_2(z)]$, for all $z \in U$ and $\varphi_2(z_0)\neq 0$ (where $U$ and $\varphi_2$ are given by Lemma \ref{defphi2}). This last property allow us to find a neighborhood of $z_0$, denoted by $W$, and a (single-valued) function $\varphi_3:U \rightarrow \C$ such that 
$\varphi_3(z)^{n+\frac{3}{2}}=\varphi_2(z)$ for all $z\in W$. We thus obtain that 
$$G(z)= [(z-z_0)\varphi_3(z)]^{n+\frac{3}{2}}. $$
Note also that $\varphi_3(z_0)\neq 0$. We let $k(z)=(z-z_0)\varphi_3(z)$. We have that $k$ is holomorphic, $k(z_0)=0$ and $k'(z_0)\neq 0$. We can thus apply an analytic version of the local inverse theorem to obtain that there exists a neighborhood of $z_0$, still denoted by $W$ and a number $r>0$ such that $k: W \rightarrow B(0,r)$ is an analytic diffeomorphism (or biholomorphism). Now we set $\Phi=k^{-1}$, we have that 
$k\circ \Phi(z)=z \in B(0,r)$, $\Phi(0)=z_0$ and $$G\circ \Phi (z)=[k(\Phi (z))]^{n+\frac{3}{2}}=z^{n+\frac{3}{2}}.$$
We then deduce that 
$$H_1\circ \Phi (z)=\text{Re}(G \circ \Phi(z))=\text{Re}(z^{n+\frac{3}{2}})$$
for all $z\in B(0,r)$ and the claim is proved. 
\end{proof}

One can check that 
\begin{eqnarray}
(\partial_xH -i\partial_yH)^2 =&(\partial_x {H_1}-i\partial_y {H_1})^ 2
= g(z)^2 
= f(z) \nonumber.
\end{eqnarray}
We thus deduce that there exists a function $\theta : W \rightarrow \{ \pm 1 \}$ such that 
\begin{equation}\label{important22}
(\partial_xh-i\partial_y h)=\theta (\partial_xH -i\partial_yH) \ \text{in} \ W.
\end{equation}
\end{proof}
\noindent We then set 
$$
W^+ = \{z \in W ; \theta (z) =+1 \}, \ \ 
W^- = \{z \in W ; \theta (z)=-1 \}. 
$$
Note that the function $\theta$ does not play the same role as in the previous section. 
This is because the function $H$ is not harmonic here. Furthermore $H$ is only lipschitz and not smooth thus we can not use the same argument as in the previous section to prove that $\theta$ is in $BV_{loc}(W \setminus \{z_0\})$. Indeed to prove this we used the fact that 
$$\theta(z)=\frac{\partial_xh-i\partial_yh}{\partial_xH-i\partial_yH}, \ \forall \ z \in W \setminus \{z_0\}$$
and we differentiated this expression in the sense of distributions, using the Leibniz rule. We can not do the same here since $\partial_xH-i\partial_yH$ is not a smooth function. \\

For this reason we work in $W \setminus \{ z \in U ; H(z)=0\}$. Thanks to Proposition \ref{levelcurvesmulti} we 
know that this set is an union of $2n+3$ connected disjoint open sets (where $n$ is defined in 
\ref{levelcurvesmulti}). We have 
\begin{equation}
W \setminus \{z \in W ; H(z)=0\} =\bigcup _{k=1}^{2n+3} W_k
\end{equation}
\noindent with $W_k$ connected and open and such that $W_k \cap W_j = \emptyset$ if $k\neq j$.
In each $W_k$, $H$ does not vanish and $(\partial_xH-i\partial_yH)$ does not vanish either. We can then find a harmonic single valued function $\tilde{H}_k$ such that 
$$|H|=\tilde{H}_k, \ in \ W_k.$$

\begin{figure}[ht!]
\begin{center}
\scalebox{0.5}{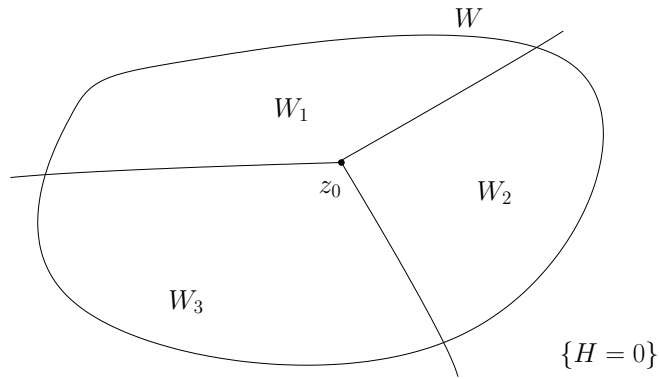}
\end{center}
\caption{Partition of $W$ in disjoint open connected subsets.}
\end{figure}
\noindent In particular $|H| \in \mathcal{C}^\infty(W_k)$ for all $1\leq k \leq 2n+3$. We are now in position to state that

\begin{claim}
The function $\theta$ is in $BV_{loc}(W_k \setminus \{z_0\})$ for all $1\leq k \leq 2n+3$.
\end{claim}

The proof of this fact is the same as the proofs of Lemma \ref{lemma1} and \ref{lemma2.1} in the previous sections. 
For $1 \leq k \leq 2n+3$ we set 
$$W_k^+ =\{ z \in W_k ; \theta(z)=+1\} .$$  
These sets are sets of locally finite perimeter in $W_k$. As in the previous sections (see Lemma \ref{colinear}) we can obtain 

\begin{claim}
The generalized outer normal $\nu_{U_k^+}$ is collinear to $\nabla h$ and $\nabla H$, $\|\partial U_k^+\|$- almost everywhere in $U_k$.
\end{claim}

We have all the ingredients to repeat the arguments of Sections 3 and 4 in each sub-domains $W_k$ and obtain

\begin{claim}\label{previousprop}
For all $1\leq k \leq 2n+3$, there exist $N_k \in \mathbb{N}$, and $N_k$ simple connected rectifiable curves $\Gamma_j^k$, $1 \leq j \leq 2n+3$ such that 
$$\partial_\star W_k^+ \setminus \partial W =\bigcup_{j=1}^{N_k} \Gamma_j^k.$$
Furthermore there exist $c_j^k$ real numbers such that $\Gamma_j^k $ is exactly a connected component of $\{ z \in U ; H(z)= c_j^k \}$.
\end{claim}

\begin{proof}[Proof of Theorem \ref{theorem3}]

Let $h$ which satisfies \eqref{C1}, \eqref{C2}, \eqref{C3}. Let $W$, $\theta$, $H$ be defined by \eqref{ImP3} in Lemma \ref{starting point3}. As before we set 
$$W \setminus \{z\in W ;H(z)= 0\} =\bigcup_{k=1}^{2n+3} W_k.$$
with $W_k$ open and connected and $W_k \cap W_j = \emptyset$ if $k\neq j$. We also set $W_k^+=\{z \in W_k; \theta(z)=+1\}$.  We use the previous Claim \ref{previousprop}
and we obtain that for all $1\leq k \leq 2n+3$, since there are only a finite number of curves $\Gamma_j^k$ such that $\partial_\star W_k^+ \setminus \partial W =\bigcup_{j=1}^{N_k} \Gamma_j^k$, with the $\Gamma_j^k$ which are connected component of  level curves of $H$ then we can find $\eta_k$ such that $$B(z_0,\eta_k) \cap \partial_\star W_k^+ \setminus \partial W \subset \{z \in W ; H(z)=0 \}.$$ We then set $\eta:= \min_{1\leq k \leq 2n+3} \eta_k$ and $V:= B(z_0,\eta)$. We have that $\theta$ is constant in each $V \cap W_k$ since $\theta$ is constant in each $B(z_0,\eta_k) \cap W_k$ from the definition of $\eta_k$. \\

We claim that $V^+=\{z\in V ; \theta (z)=+1\}$ is a set of finite perimeter in $V$. Indeed we have that   $\partial V^+ \setminus \partial V \subset \{z\in V; H(z)=0\}$ (here we use the topological boundary $\partial V^+$) and $\mathcal{H}^1 (\{z\in V; H(z)=0\}) <+\infty$ from the last point of Lemma \ref{starting point3}. Then applying Proposition 3.62 of \cite{AmbrosioFuscoPallara} we deduce that $V^+$ is a set of finite perimeter. \\

Note that in $V$ the function $\theta H$ is continuous since $H=0$ at the discontinuity points of $\theta$. Then $\theta H$ is in $H^1(V)$ since $H$ is in $H^1$. Computing $\nabla (\theta H)$ in the sense of distributions we  obtain that $\nabla (\theta H)=\theta \nabla H$ and it comes 
$$\nabla (h-\theta H )=0 \ \text{in} \ V.$$
this proves that $h-\theta H$ is constant in $V$, but evaluating this constant in $z_0$ we find that 
$$h=\theta H \ \text{in} \ V.$$




\end{proof}

\section{Appendix: On the stationary harmonic functions}

This appendix is devoted to elementary results on stationary harmonic functions. These results are stated without proof in the introduction of this paper. The original definition of stationary harmonic function is the following:

\begin{definition}
A function $h$ in $H^1(\Omega)$ is \textit{stationary harmonic} if for any family of diffeomorphisms $\phi_t$ of $\Omega$ such that $\phi_0=\text{Id}$ we have 
\begin{equation}\nonumber
\frac{d}{dt}|_{t=0} E(h \circ \phi_t)=0,
\end{equation}
\noindent where $E(h)=\frac{1}{2}\int_\Omega |\nabla h|^2dx$ is the Dirichlet energy.
\end{definition}

\noindent As shown by the following proposition we used an equivalent characterization.
\begin{proposition}
A function $h$ is stationary harmonic if and only if $\dive T_h=0$ in the sense of distributions, where 
$$T_h=\begin{pmatrix}
\frac{1}{2}\left[(\partial_y h) ^2-(\partial_x h)^2\right] & -\partial_x h \partial_y h \\
-\partial_x h \partial_y h & \frac{1}{2}\left[(\partial_x h) ^2-(\partial_y h)^2\right].
\end{pmatrix} $$
\end{proposition}
\begin{proof}
We first note that $\dive(T_h)=0$ in the sense of distributions if and only if 
$$\int_\Omega \langle T_h,D\eta \rangle =0, \ \ \ \forall \eta \in \mathcal{C}^{\infty}_c(\Omega,\R^2), $$
where $D\eta$ denotes the differential of $\eta$ (which is a $2\times 2$ matrix) and $\langle A,B \rangle=\tra(^tAB)$ denotes the inner product of two matrices. Let $\phi_t(x)=x+t\eta(x)$ with $\eta \in \mathcal{C}^\infty_c(\Omega,\R^2)$, if $t$ is small enough $\phi_t$ is a diffeomorphism. Let $h_t:=h\circ \phi_t$, we have 
$$\nabla h_t(x)= (I+tD\eta(x)).\nabla h(x+t\eta(x))$$
$$|\nabla h_t(x)|^2=|\nabla h(x+t\eta(x))|^2+2t\langle \nabla h(x+t\eta(x)), [D\eta(x).\nabla h(x+t\eta(x))]\rangle +o(t) $$
Then 
\begin{eqnarray}
\frac{1}{2} \int_{\Omega} |\nabla h_t(x)|^2dx& =&\frac{1}{2}\int_\Omega |\nabla h(x+t\eta(x))|^2dx + \nonumber \\ & +&t\int_\Omega \langle \nabla h(x+t\eta(x)), \left[D\eta(x).\nabla h(x+t\eta(x))\right]\rangle dx+o(t). \nonumber
\end{eqnarray}
We can make the following change of variables $y=x+t\eta(x) \Leftrightarrow x=y-t\eta(x) \Rightarrow x=y-t\eta(y)+o(t)$ (the last implication holds because $\eta(x)=\eta(y)+o(1)$ when $t$ goes to $0$). 
We also have 
$$\det\left[D(y-t\eta(y)+o(t))\right]=\det(I-D\eta +o(t))=1-t \tra D\eta +o(t)$$ because 
$\det(I+tA)=1+ t\tra(A)+o(t)$. Then 
\begin{eqnarray}
E(h_t)&=&\frac{1}{2}\int_\Omega |\nabla h(y)|^2dy-\frac{t}{2}\int_\Omega |\nabla h(y)|^2\tra D\eta(y)dy + \nonumber \\ &+&t\int_\Omega \langle \nabla h(y), \left[D\eta(y)\nabla h(y)\right] \rangle dy +o(t). \nonumber
\end{eqnarray}
Hence 
\begin{equation}
\frac{d}{dt}|_{|t=0}E(h_t)=0 \nonumber
\end{equation}
is equivalent to 
\begin{equation}
\int_\Omega[-\frac{1}{2}|\nabla h(y)|^2 \tra D\eta(y)+\langle \nabla h(y),[ D\eta (y).\nabla h(y)]\rangle dy=0. \nonumber
\end{equation}

\noindent But $$\frac{1}{2}|\nabla h(y)|^2\tra D\eta(y)=\langle \frac{1}{2}|\nabla h(y)|^2I,D\eta(y)\rangle$$ and \\
$$\langle \nabla h(y) , D\eta(y)\nabla h(y)\rangle =\langle \nabla h(y)^t\nabla h(y),D\eta(y)\rangle$$
with $\nabla h(y)^t\nabla h(y)= \begin{pmatrix}
(\partial_x h)^2 & \partial_xh \partial_yh \\
\partial_yh\partial_xh & (\partial_y h) ^2
\end{pmatrix}$.
We can then conclude that $$\frac{d}{dt}|_{|t=0}E(h_t)=0 \Leftrightarrow \int_\Omega <T_h,D\eta>=0$$ which is equivalent to $\dive(T_h)=0$ with $T_h=\nabla h ^t\nabla h-\frac{1}{2}|\nabla h|^2I$.
\end{proof}
\noindent The equation \eqref{2} can also be interpreted in terms of holomorphic functions

\begin{proposition}\label{Hopf}
The condition $\dive(T_h)=0$ is equivalent to $\omega_h:= |\partial_xh|^2-|\partial_y h|^2-2i\partial_xh\partial_y h$ is holomorphic in $\Omega$. 
\end{proposition}

\begin{proof}
$$\dive(T_h)=0 \Leftrightarrow \left \{ 
\begin{array}{rcll}
\partial_x (\partial_xh ^2-\partial_y h ^2)&=& \partial_y(-2\partial_x \partial_yh) \\
\partial_y (\partial_xh ^2-\partial_y h ^2)&=& -\partial_x(-2\partial_x \partial_yh).
\end{array}
\right.$$
These are the Cauchy-Riemann equations for $\omega_h$ written in the sense of distributions. We can rewrite them as $\partial_{\bar{z}}\omega_h=0$ where $\partial_{\bar{z}}=\frac{1}{2}(\partial_x+i\partial_y)$.
The operator $\partial_{\bar{z}}$ is elliptic and hence the elliptic regularity theory shows that $\omega_h$ is smooth and then holomorphic because it satisfies the Cauchy-Riemann equations.
\end{proof}

\begin{proposition}
If $h$ is harmonic in $\Omega$ then $h$ is stationary harmonic in $\Omega$.
\end{proposition}

\begin{proof}
Assume that $\Delta h=0$ in $\Omega$. Recall that $\Delta v =4 \partial_{\bar{z}} \partial_z v$ and let us compute 
\begin{equation}\nonumber
\begin{array}{rclll}
\partial_{\bar{z}} \omega _h&=&4 \partial_{\bar{z}}[\partial_z h]^2 &=& 8\partial_zh \partial_{\bar{z}}\partial_z h \\
&=& 8\partial_zh \Delta h &=&0.
\end{array}
\end{equation}
Hence $\partial_{\bar{z}}\omega_h(z)=0$, that is $\omega_h$ is holomorphic.\\
\end{proof}

The converse of the previous proposition is not true. However if $h$ is a stationary harmonic functions which statisfies the hypothesis \eqref{1}, \eqref{3}, \eqref{2} with $\mu \in L^p$, $p>1$, then, using the same methods as in \cite{SandierSerfaty} Chapter 13, one can show that $h$ is harmonic.

\begin{proposition}\label{regularity}
If $\mu$ is in $L^p(\Omega)$ for some $p>1$ and satisfies \eqref{2},\eqref{3} then $\mu=0$.
\end{proposition}

\begin{proof}
Let $\mu$ be in $L^p(\Omega)$ for some $p>1$ and such that $\dive(\T)=0$ and $\Delta \h=\mu$. 
Let $\rho_n$ be a regularizing kernel, we set $\mu_n:=\mu \ast \rho_n$, $h_n:= \h \ast \rho_n$ and 
$$T_n:= \frac{1}{2}\begin{pmatrix}
\partial_yh_n^2-\partial_xh_n^2 & -2\partial_xh_n\partial_yh_n \\
-2\partial_xh_n\partial_yh_n & \partial_yh_n^2-\partial_xh_n^2
\end{pmatrix}.$$
One has $\mu_n \rightarrow \mu$ in $L^p(\Omega)$, and because $\nabla \h$ is in $L_{loc}^\infty(\Omega)$ one also has $\nabla h_n \rightarrow \nabla \h$ in $L_{loc}^q(\Omega)$, for all $q \in [1,+\infty[$. Then
$$\mu_n\nabla h_n \rightarrow \mu \nabla \h, \ \text{in} \ L_{loc}^1(\Omega)$$
and
$$T_n \rightarrow T_\mu, \ \text{in} \ L_{loc}^1(\Omega).$$
The last equation implies that $\dive(T_n) \rightarrow \dive(T_\mu)=0$ and $\mu_n\nabla h_n \rightarrow \mu \nabla \h $ in the sense of distributions. However $\dive(T_n)=-\Delta h_n \nabla h_n=\mu_n\nabla h_n$ thus 
$\mu\nabla \h=\lim_{n\rightarrow +\infty} \dive(T_n)=0$ in $L^1_{loc}(\Omega)$ and almost everywhere.
From a  property of Sobolev functions we have $\Delta \h=0$ a.e. on the set $F=\{\nabla\h=0\}$, thus $\mu=0$ a.e. on $F$ and $\mu=0$ on $\Omega \setminus F$ hence $\mu=0$ on $\Omega$.
\end{proof}
\bigskip
\footnotesize
\noindent\textit{Acknowledgments.} I would like to thank my Ph.D. advisor E.Sandier for his support and useful comments on this subject. I would also like to thank J.M. Delort and S.Masnou for interesting discussions on this work, and X.Lamy for valuable comments which helped to improve the presentation of the paper.

\nocite*

\bibliographystyle{plain}
\bibliography{biblioreg}

\end{document}

%% file: Figure1.pdf_t
\begin{picture}(0,0)%
\includegraphics{Figure1.pdf}%
\end{picture}%
\setlength{\unitlength}{4144sp}%
\begingroup\makeatletter\ifx\SetFigFont\undefined%
\gdef\SetFigFont#1#2#3#4#5{%
  \reset@font\fontsize{#1}{#2pt}%
  \fontfamily{#3}\fontseries{#4}\fontshape{#5}%
  \selectfont}%
\fi\endgroup%
\begin{picture}(5466,4970)(4468,-7946)
\put(4861,-3931){\makebox(0,0)[lb]{\smash{{\SetFigFont{20}{24.0}{\familydefault}{\mddefault}{\updefault}{\color[rgb]{0,0,0}$V$}%
}}}}
\put(7291,-5731){\makebox(0,0)[lb]{\smash{{\SetFigFont{20}{24.0}{\familydefault}{\mddefault}{\updefault}{\color[rgb]{0,0,0}$z_0$}%
}}}}
\put(9901,-5731){\makebox(0,0)[lb]{\smash{{\SetFigFont{20}{24.0}{\familydefault}{\mddefault}{\updefault}{\color[rgb]{0,0,0}$\text{supp} \mu=\{H=0\}$}%
}}}}
\end{picture}%

%% file: Figure2.pdf_t
\begin{picture}(0,0)%
\includegraphics{Figure2.pdf}%
\end{picture}%
\setlength{\unitlength}{4144sp}%
\begingroup\makeatletter\ifx\SetFigFont\undefined%
\gdef\SetFigFont#1#2#3#4#5{%
  \reset@font\fontsize{#1}{#2pt}%
  \fontfamily{#3}\fontseries{#4}\fontshape{#5}%
  \selectfont}%
\fi\endgroup%
\begin{picture}(6147,5964)(4396,-9343)
\put(7561,-6541){\makebox(0,0)[lb]{\smash{{\SetFigFont{20}{24.0}{\familydefault}{\mddefault}{\updefault}{\color[rgb]{0,0,0}$0$}%
}}}}
\put(9901,-4471){\makebox(0,0)[lb]{\smash{{\SetFigFont{20}{24.0}{\familydefault}{\mddefault}{\updefault}{\color[rgb]{0,0,0}$D_1$}%
}}}}
\put(4771,-4381){\makebox(0,0)[lb]{\smash{{\SetFigFont{20}{24.0}{\familydefault}{\mddefault}{\updefault}{\color[rgb]{0,0,0}$D_2$}%
}}}}
\put(4411,-6091){\makebox(0,0)[lb]{\smash{{\SetFigFont{20}{24.0}{\familydefault}{\mddefault}{\updefault}{\color[rgb]{0,0,0}$V$}%
}}}}
\end{picture}%

%% file: Figure3.pdf_t
\begin{picture}(0,0)%
\includegraphics{Figure3.pdf}%
\end{picture}%
\setlength{\unitlength}{4144sp}%
\begingroup\makeatletter\ifx\SetFigFont\undefined%
\gdef\SetFigFont#1#2#3#4#5{%
  \reset@font\fontsize{#1}{#2pt}%
  \fontfamily{#3}\fontseries{#4}\fontshape{#5}%
  \selectfont}%
\fi\endgroup%
\begin{picture}(5725,4419)(4648,-7339)
\put(7021,-5731){\makebox(0,0)[lb]{\smash{{\SetFigFont{20}{24.0}{\familydefault}{\mddefault}{\updefault}{\color[rgb]{0,0,0}$z_0$}%
}}}}
\put(5401,-4921){\makebox(0,0)[lb]{\smash{{\SetFigFont{20}{24.0}{\familydefault}{\mddefault}{\updefault}{\color[rgb]{0,0,0}$\text{supp} \mu$}%
}}}}
\put(5941,-3211){\makebox(0,0)[lb]{\smash{{\SetFigFont{20}{24.0}{\familydefault}{\mddefault}{\updefault}{\color[rgb]{0,0,0}$V$}%
}}}}
\end{picture}%

%% file: figure5.pdf_t
\begin{picture}(0,0)%
\includegraphics{figure5.pdf}%
\end{picture}%
\setlength{\unitlength}{4144sp}%
\begingroup\makeatletter\ifx\SetFigFont\undefined%
\gdef\SetFigFont#1#2#3#4#5{%
  \reset@font\fontsize{#1}{#2pt}%
  \fontfamily{#3}\fontseries{#4}\fontshape{#5}%
  \selectfont}%
\fi\endgroup%
\begin{picture}(6591,5938)(4558,-7800)
\put(4726,-2896){\makebox(0,0)[lb]{\smash{{\SetFigFont{20}{24.0}{\familydefault}{\mddefault}{\updefault}{\color[rgb]{0,0,0}$B_{R'}$}%
}}}}
\put(10801,-6091){\makebox(0,0)[lb]{\smash{{\SetFigFont{20}{24.0}{\familydefault}{\mddefault}{\updefault}{\color[rgb]{0,0,0}$\Gamma_j$}%
}}}}
\put(7786,-5236){\makebox(0,0)[lb]{\smash{{\SetFigFont{20}{24.0}{\familydefault}{\mddefault}{\updefault}{\color[rgb]{0,0,0}$z_0$}%
}}}}
\put(7606,-2266){\makebox(0,0)[lb]{\smash{{\SetFigFont{20}{24.0}{\familydefault}{\mddefault}{\updefault}{\color[rgb]{0,0,0}$+$}%
}}}}
\put(7606,-3211){\makebox(0,0)[lb]{\smash{{\SetFigFont{20}{24.0}{\familydefault}{\mddefault}{\updefault}{\color[rgb]{0,0,0}$-$}%
}}}}
\put(7831,-4246){\makebox(0,0)[lb]{\smash{{\SetFigFont{20}{24.0}{\familydefault}{\mddefault}{\updefault}{\color[rgb]{0,0,0}$+$}%
}}}}
\put(6751,-5416){\makebox(0,0)[lb]{\smash{{\SetFigFont{20}{24.0}{\familydefault}{\mddefault}{\updefault}{\color[rgb]{0,0,0}$-$}%
}}}}
\put(7921,-6361){\makebox(0,0)[lb]{\smash{{\SetFigFont{20}{24.0}{\familydefault}{\mddefault}{\updefault}{\color[rgb]{0,0,0}$+$}%
}}}}
\put(8146,-7441){\makebox(0,0)[lb]{\smash{{\SetFigFont{20}{24.0}{\familydefault}{\mddefault}{\updefault}{\color[rgb]{0,0,0}$-$}%
}}}}
\end{picture}%

%% file: figure6.pdf_t
\begin{picture}(0,0)%
\includegraphics{figure6.pdf}%
\end{picture}%
\setlength{\unitlength}{4144sp}%
\begingroup\makeatletter\ifx\SetFigFont\undefined%
\gdef\SetFigFont#1#2#3#4#5{%
  \reset@font\fontsize{#1}{#2pt}%
  \fontfamily{#3}\fontseries{#4}\fontshape{#5}%
  \selectfont}%
\fi\endgroup%
\begin{picture}(7716,7536)(3343,-7609)
\put(10261,-1411){\makebox(0,0)[lb]{\smash{{\SetFigFont{20}{24.0}{\familydefault}{\mddefault}{\updefault}{\color[rgb]{0,0,0}$B_{R'}$}%
}}}}
\put(10711,-6541){\makebox(0,0)[lb]{\smash{{\SetFigFont{20}{24.0}{\familydefault}{\mddefault}{\updefault}{\color[rgb]{0,0,0}$\{H=0\}$}%
}}}}
\put(7066,-4336){\makebox(0,0)[lb]{\smash{{\SetFigFont{20}{24.0}{\familydefault}{\mddefault}{\updefault}{\color[rgb]{0,0,0}$z_0$}%
}}}}
\put(6976,-1411){\makebox(0,0)[lb]{\smash{{\SetFigFont{20}{24.0}{\familydefault}{\mddefault}{\updefault}{\color[rgb]{0,0,0}$+$}%
}}}}
\put(8191,-3931){\makebox(0,0)[lb]{\smash{{\SetFigFont{20}{24.0}{\familydefault}{\mddefault}{\updefault}{\color[rgb]{0,0,0}$+$}%
}}}}
\put(7291,-4741){\makebox(0,0)[lb]{\smash{{\SetFigFont{20}{24.0}{\familydefault}{\mddefault}{\updefault}{\color[rgb]{0,0,0}$-$}%
}}}}
\put(6301,-4111){\makebox(0,0)[lb]{\smash{{\SetFigFont{20}{24.0}{\familydefault}{\mddefault}{\updefault}{\color[rgb]{0,0,0}$+$}%
}}}}
\put(10441,-3931){\makebox(0,0)[lb]{\smash{{\SetFigFont{20}{24.0}{\familydefault}{\mddefault}{\updefault}{\color[rgb]{0,0,0}$-$}%
}}}}
\put(7246,-6496){\makebox(0,0)[lb]{\smash{{\SetFigFont{20}{24.0}{\familydefault}{\mddefault}{\updefault}{\color[rgb]{0,0,0}$+$}%
}}}}
\put(7336,-7306){\makebox(0,0)[lb]{\smash{{\SetFigFont{20}{24.0}{\familydefault}{\mddefault}{\updefault}{\color[rgb]{0,0,0}$-$}%
}}}}
\put(3961,-4201){\makebox(0,0)[lb]{\smash{{\SetFigFont{20}{24.0}{\familydefault}{\mddefault}{\updefault}{\color[rgb]{0,0,0}$-$}%
}}}}
\put(8821,-2896){\makebox(0,0)[lb]{\smash{{\SetFigFont{20}{24.0}{\familydefault}{\mddefault}{\updefault}{\color[rgb]{0,0,0}$\rho$}%
}}}}
\put(7156,-3391){\makebox(0,0)[lb]{\smash{{\SetFigFont{20}{24.0}{\familydefault}{\mddefault}{\updefault}{\color[rgb]{0,0,0}$-$}%
}}}}
\put(4321,-601){\makebox(0,0)[lb]{\smash{{\SetFigFont{20}{24.0}{\familydefault}{\mddefault}{\updefault}{\color[rgb]{0,0,0}$\Gamma_j^\rho$}%
}}}}
\end{picture}%

%% file: figure7.pdf_t
\begin{picture}(0,0)%
\includegraphics{figure7.pdf}%
\end{picture}%
\setlength{\unitlength}{4144sp}%
\begingroup\makeatletter\ifx\SetFigFont\undefined%
\gdef\SetFigFont#1#2#3#4#5{%
  \reset@font\fontsize{#1}{#2pt}%
  \fontfamily{#3}\fontseries{#4}\fontshape{#5}%
  \selectfont}%
\fi\endgroup%
\begin{picture}(7118,4488)(3904,-6688)
\put(9226,-2491){\makebox(0,0)[lb]{\smash{{\SetFigFont{20}{24.0}{\familydefault}{\mddefault}{\updefault}{\color[rgb]{0,0,0}$W$}%
}}}}
\put(7066,-3571){\makebox(0,0)[lb]{\smash{{\SetFigFont{20}{24.0}{\familydefault}{\mddefault}{\updefault}{\color[rgb]{0,0,0}$W_1$}%
}}}}
\put(9496,-4561){\makebox(0,0)[lb]{\smash{{\SetFigFont{20}{24.0}{\familydefault}{\mddefault}{\updefault}{\color[rgb]{0,0,0}$W_2$}%
}}}}
\put(7606,-4471){\makebox(0,0)[lb]{\smash{{\SetFigFont{20}{24.0}{\familydefault}{\mddefault}{\updefault}{\color[rgb]{0,0,0}$z_0$}%
}}}}
\put(5806,-5866){\makebox(0,0)[lb]{\smash{{\SetFigFont{20}{24.0}{\familydefault}{\mddefault}{\updefault}{\color[rgb]{0,0,0}$W_3$}%
}}}}
\put(10486,-6496){\makebox(0,0)[lb]{\smash{{\SetFigFont{20}{24.0}{\familydefault}{\mddefault}{\updefault}{\color[rgb]{0,0,0}$\{H=0\}$}%
}}}}
\end{picture}%